\documentclass{amsart}

\usepackage{a4wide}
\usepackage{pdfsync,changes} 
\usepackage{amsmath}
\usepackage{amssymb}
\usepackage{amsfonts}
\usepackage{amsthm}
\usepackage{todonotes}
\usepackage{subfigure}
\usepackage[all,cmtip]{xy}

\newtheorem{lemma}{Lemma}[section]
\newtheorem{theorem}[lemma]{Theorem}

\newtheorem{proposition}[lemma]{Proposition}
\newtheorem{corollary}[lemma]{Corollary}
\newtheorem{questions}[lemma]{Questions}
\newcommand{\thistheoremname}{}
\newtheorem*{genericthm}{\thistheoremname}

\theoremstyle{remark}

\newtheorem{remark}[lemma]{Remark}

\newtheorem{definition}[lemma]{Definition}

\numberwithin{equation}{section}

\newcommand{\N}{\mathbb{N}}

\newcommand{\Q}{\mathbb{Q}}
\newcommand{\R}{\mathbb{R}}

\newcommand{\ompr}{{\widehat{\omega}}}
\newcommand{\subw}[1]{\Sigma_{#1}}
\newcommand{\sipr}{{\widehat{\Sigma}}}

\newcommand{\sighat}{{\widehat{\sigma}}}
\newcommand{\That}{{\widehat{T}}}
\newcommand{\smallthat}{{\widehat{t}}}
\newcommand{\Shat}{{\widehat{S}}}
\newcommand{\smallshat}{{\widehat{s}}}

\newcommand{\cA}{\mathcal{A}} 
\newcommand{\cS}{\mathcal{S}} 
\newcommand{\cL}{\mathcal{L}} 

\newcommand{\bsigma}{\boldsymbol{\sigma}}

\newcommand{\bu}{\mathbf{u}}

\newcommand{\bx}{\mathbf{x}}

\makeatletter
\newcommand{\labitem}[2]{\def\@itemlabel{#1}\item\def\@currentlabel{#1}\label{#2}}
\makeatother

\definechangesauthor[color=orange]{X}
\definechangesauthor[color=magenta]{Y}
\definechangesauthor[color=cyan]{W}
\definechangesauthor[color=red]{J}
\definechangesauthor[color=blue]{V}

\title[Generalizations of Sturmian sequences]{Generalizations of Sturmian sequences associated with $\boldsymbol{N}$-continued fraction algorithms}

\author{Niels Langeveld, Luc\'ia Rossi, and J\"org M. Thuswaldner}

\address{Chair of Mathematics and Statistics, University of Leoben, Franz-Josef-Strasse 18, A-8700 Leoben, Austria}
\email{niels.langeveld@unileoben.ac.at}
\email{lucia.rossi@unileoben.ac.at}
\email{joerg.thuswaldner@unileoben.ac.at}

\date{\today}

    \thanks{The doctoral position of the second author is supported by the Austrian Science Fund (FWF) as part of Discrete Mathematics Doctoral Program, project W1230. The third author is supported by the bilateral project I 5554 funded by the FWF and RSF}
\begin{document}
	
	\begin{abstract}
		Given a positive integer $N$ and $x\in[0,1]\setminus\Q$, an {\it $N$-continued fraction expansion} of $x$ is defined analogously to the classical continued fraction expansion, but with the numerators being all equal to $N$. Inspired by Sturmian sequences, we introduce the {\it $N$-continued fraction sequences} $\omega(x,N)$ and $\ompr(x,N)$, which are related to the $N$-continued fraction expansion of $x$. They are infinite words over a two letter alphabet obtained as the limit of a directive sequence of certain substitutions, hence they are {\it $S$-adic sequences}. When $N=1$, we are in the case of the classical continued fraction algorithm, and obtain the well-known Sturmian sequences. We show that $\omega(x,N)$ and $\ompr(x,N)$ are $C$-balanced for some explicit values of $C$ and compute their factor complexity function. We also obtain uniform word frequencies and  deduce unique ergodicity of the associated subshifts. Finally, we provide a Farey-like map for $N$-continued fraction expansions, which provides an additive version of $N$-continued fractions, for which we prove ergodicity and give the invariant measure explicitly.
	\end{abstract}
	
	\maketitle

	\section{Introduction}
	\subsection{The setting}
	Let $N$ be a positive integer. In this article, we will introduce a family of sequences over the alphabet $\{0,1\}$ that we call {\it $N$-continued fraction sequences}. As their name indicates, these sequences are related to the so-called {\it $N$-continued fraction algorithms} that were introduced by Burger {\it et al}.~\cite{BGKWY}. Given a real number $x\in[0,1]$, an {\it $N$-continued fraction expansion}  (or NCF expansion, for short) of $x$ is an expansion of the form 
	\[	
	x= \frac{N}{\displaystyle d_1+\frac{N}{\displaystyle d_2+\ddots}},
	\] 
	with {\it $N$-continued fraction digits} $d_n\ge 1$. If $N>1$, it turns out that there exist infinitely many different NCF expansions of $x$ (see \cite{DKW}). Though, when we impose that $d_n\ge N$, we find a unique infinite expansion for all irrational numbers and exactly two finite expansions for rational numbers.  These expansions are called the greedy NCF expansions. 
	
	$N$-continued fraction expansions and $N$-continued fraction sequences form a natural generalization of a well known setting. Indeed, if $N=1$, the $1$-continued fraction algorithm is just the classical continued fraction algorithm and $1$-continued fraction sequences will turn out to be {\it Sturmian sequences}, which have been studied extensively (see for instance the survey \cite{A2}). We recall that a sequence $\omega$ over the alphabet $\{0,1\}$ is called Sturmian if, for any given $n\in \N$, there exist exactly $n+1$ pairwise different factors (subwords) of $\omega$ of length $n$. More formally, we say that the {\it factor complexity function} $p_\omega$ of $\omega$ satisfies $p_\omega(n)=n+1$. It is easy to see that Sturmian sequences are the non eventually periodic sequences with the smallest possible complexity (see {\it e.g.}~\cite[Corollary~4.3.2]{CN}). A Sturmian sequence $\omega$ has the property that, given any two factors $u$ and $v$  of $\omega$ of the same length and a letter $a\in \{0,1\}$, the number of occurrences of $a$ in $u$ differs from the number of occurrences of $a$ in $v$ by at most one. This property is called {\it $1$-balance} (or just {\it balance}) and characterizes Sturmian sequences. More precisely, the Sturmian sequences are exactly the non eventually periodic balanced sequences over two letters ({\it cf. e.g.}~\cite[Theorem~6.1.8]{A2}). Moreover, Morse and Hedlund~\cite{Morse&Hedlund:1940} as well as Coven and Hedlund~\cite{Coven&Hedlund:1973} discovered a connection between Sturmian sequences and rotations by an irrational number $\alpha$. Rauzy found a very elegant proof for this connection that relates Sturmian sequences to $S$-adic sequences and to the classical continued fraction algorithm (see \cite{Arnoux-Rauzy:91} and Rauzy's earlier papers~\cite{Rauzy:77,Rauzy:79}). 
	
	The aim of the present article is to carry over some of these properties of Sturmian sequences to the more general setting of $N$-continued fraction sequences.

	\subsection{Outline of the paper}	
	In Section~\ref{sec:prel} we introduce some basic notions related to NCF expansions, alphabets and words, substitutions, and {\it $S$-adic} sequences. 
	
	Section~\ref{sec:NCFwords} is devoted to the definition of NCF sequences and their duals and their relation to NCF algorithms. Inspired by the Sturmian substitutions and their duals (see {\it e.g.}~\cite{AF:01}), given $N\geq1$ we associate two sets $\cS$ and $\widehat \cS$ of substitutions on the alphabet $\cA= \{0,1\}$ to the NCF algorithm (see Definition~\ref{def:subst}). By applying the NCF algorithm to $x\in[0,1]\setminus\Q$, we can then associate certain {\it directive sequences} $\bsigma_x$ and $\widehat{\bsigma}_x$ of substitutions taken from $\cS$ and $\widehat \cS$, respectively. Indeed, these sequences are defined in a natural way in terms of the NCF digits of $x$ and can be regarded as a combinatorial version of the NCF expansion of $x$. The sequences $\bsigma_x$ and $\widehat{\bsigma}_x$ can then be used to define the {\it $S$-adic sequences} (or  limit sequences) $\omega(x,N)$ and $\ompr(x,N)$, respectively, which are the NCF sequences we are interested in and their duals. We will show that the ratio of the letter frequencies in $\omega(x,N)$ converges to $x$. This is related to the existence of a generalized right eigenvector for the directive sequence $\bsigma_x$. 
	
	In Section \ref{sec:balance}, we prove balance properties of the NCF sequences and their duals that generalize the $1$-balance of Sturmian sequences. In particular, we show that $\omega(x,N)$ is $N^2$-balanced and $\ompr(x,N)$ is $N$-balanced for each $x$. Moreover, for each $K\geq N$ we consider the set $W_{K,N}$ of all irrational $x\in[0,1]$ whose NCF digits are greater than or equal to $K$. We show that $\ompr(x,N)$ is $ C $-balanced and $\omega(x,N)$ is $ N \cdot C $-balanced for all $x \in W_{K,N} $ for some explicitly given constant $C=C(K,N)$ that approaches $2$ if $K$ tends to infinity (if $N\ge 2$ is fixed). We also give lower bounds for the balance constants. 
	
	In Section \ref{sec:complexity}, we study the factor complexity of $\omega(x,N)$ and $\ompr(x,N)$, which also generalizes the complexity function of Sturmian sequences. This is done by characterizing {\it special factors}  of these sequences. More specifically, if $p_\omega$ and $p_\ompr$ are the respective complexity functions, we show that $p_{\omega}(n+1) - p_{\omega}(n) $ and $p_{\ompr}(n+1) - p_{\ompr}(n)$ can only take the values $1$ or $2$, which implies that $p_\omega(n) \leq 2n$ and $p_\ompr(n) \leq 2n$ for all $n \in \N$. We specify for each $n\in\N$ which of these two values is attained. this allows us to give an explicit formula of the factor complexity, and in fact we give it in terms of the {\it convergents} $c_n=\frac{p_n}{q_n}$ of the NCF expansion of $x$, {\it i.e.}, the rationals obtained by truncating the $N$-continued fraction expansion after $n$ steps. From the factor complexity function, we deduce uniform word frequencies for both families of sequences, and unique ergodicity of the topological dynamical systems given by the respective subshifts.
	
	In Section \ref{sec:dynamics}, we state properties of growth rate and entropy, and give a Farey-like map for greedy NCF expansions together with its invariant measure. Finally, we pose some open questions. 
	
	\section{Preliminaries}\label{sec:prel}
	
	In this section, we state some basic definitions and properties regarding $N$-continued fractions, substitutions, and $S$-adic sequences.
	
	\subsection{$\boldsymbol{N}$-continued fraction expansions} 
	As a variation on the regular continued fraction algorithm, Burger {\it et al.} introduced $N$-continued fraction expansions in \cite{BGKWY}. These are expansions of the form
	\[
	x=\frac{N}{\displaystyle d_1+\frac{N}{\displaystyle d_2+\ddots}}.
	\]
	For $N=1$, we find back the regular continued fraction algorithm. For the case $N \geq 2$, these continued fractions share some properties with the regular ones but there are also differences worth mentioning. For example, in contrast to the regular continued fraction expansions, any number has infinitely many different NCF expansions, see \cite{AW,DKW}. Also the behavior of quadratic irrationals seems to be  very different. For regular continued fractions, we know that any quadratic irrational has a purely or eventually periodic expansion. In \cite{BGKWY} it is proven that for every quadratic irrational number there exist infinitely many eventually periodic NCF expansions with period-length 1. On the other hand, for a fixed $N \geq 2$ it seems that there are quadratic irrational numbers with aperiodic NCF expansions, see \cite{DKW}. Another difference is that for Lebesgue almost all $x\in[0,1]$ the regular continued fraction expansion has arbitrarily large digits, but for NCF expansions we can find for every $x$ an NCF expansion such that the digits are bounded, see \cite{KL}.

	In this article, we look at the {\it greedy NCF expansion} obtained from the map $T_N$.
	Fix $N \geq 1$ and define $T_N:[0,1]\rightarrow[0,1]$ as 
	\begin{equation}\label{mapTN}
		T_N(x)=
		\begin{cases}	
			\frac{N}{x}-\left\lfloor \frac{N}{x} \right\rfloor&x\neq 0,\\
			0	& x=0,
		\end{cases}
	\end{equation}
	see Figure \ref{fig:T2} for examples. Set $d_1(x)=\big\lfloor \frac{N}{x} \big\rfloor$ and $d_n(x)=d_1(T^{n-1}_N(x))$ whenever $T^{n-1}_N(x)\neq 0$.
	For $x$ we find 
	\begin{eqnarray*}
		x&=& \frac{N}{d_1(x)+T_N(x)}\\
		&=& \frac{N}{\displaystyle d_1(x)+\frac{N}{d_2(x)+T_N^2(x)}}\\
		&=&  \frac{N}{\displaystyle d_1(x)+\frac{N}{\displaystyle d_2(x)+\ddots}}.
	\end{eqnarray*}
	
	\begin{figure}[ht]
		\centering
		\subfigure{\begin{tikzpicture}[scale=5]
				\draw(0,0)node[below]{\small $0$}--(1,0)node[below]{\small $1$}--(1,1)--(0,1)node[left]{\small $1$}--(0,0);

				\draw(0.15,0.5)node{\Huge $\cdots$};
				
				\draw[thick, purple!50!black, smooth, samples =20, domain=2/3:1] plot(\x,{2/\x-2});
				\draw[thick, purple!50!black, smooth, samples =20, domain=2/4:2/3] plot(\x,{2/\x-3});
				\draw[thick, purple!50!black, smooth, samples =20, domain=2/5:2/4] plot(\x,{2/\x-4});
				\draw[thick, purple!50!black, smooth, samples =20, domain=2/6:2/5] plot(\x,{2/\x-5});
				\draw[thick, purple!50!black, smooth, samples =20, domain=2/7:2/6] plot(\x,{2/\x-6});
				
				\draw[dotted](2/3,0)node[below]{\small $\frac{2}{3}$}--(2/3,1);
				\draw[dotted](0.5,0)node[below]{\small $\frac{1}{2}$}--(0.5,1);
				\draw[dotted](2/5,0)node[below]{\small $\frac{2}{5}$}--(2/5,1);
				\draw[dotted](1/3,0)node[below]{\small $\frac{1}{3}$}--(1/3,1);
				\draw[dotted](2/7,0)node[below]{\small $\frac{2}{7}$}--(2/7,1);
		\end{tikzpicture}}\hskip 1.5cm
		\subfigure{\begin{tikzpicture}[scale=5]
				\draw(0,0)node[below]{\small $0$}--(1,0)node[below]{\small $1$}--(1,1)--(0,1)node[left]{\small $1$}--(0,0);

				\draw(0.2,0.5)node{\Huge $\cdots$};
				
				\draw[thick, purple!50!black, smooth, samples =20, domain=5/6:1] plot(\x,{5/\x-5});
				\draw[thick, purple!50!black, smooth, samples =20, domain=5/7:5/6] plot(\x,{5/\x-6});
				\draw[thick, purple!50!black, smooth, samples =20, domain=5/8:5/7] plot(\x,{5/\x-7});
				\draw[thick, purple!50!black, smooth, samples =20, domain=5/9:5/8] plot(\x,{5/\x-8});
				\draw[thick, purple!50!black, smooth, samples =20, domain=5/10:5/9] plot(\x,{5/\x-9});
				\draw[thick, purple!50!black, smooth, samples =20, domain=5/11:5/10] plot(\x,{5/\x-10});
				\draw[thick, purple!50!black, smooth, samples =20, domain=5/12:5/11] plot(\x,{5/\x-11});
				\draw[thick, purple!50!black, smooth, samples =20, domain=5/13:5/12] plot(\x,{5/\x-12});
				\draw[thick, purple!50!black, smooth, samples =20, domain=5/14:5/13] plot(\x,{5/\x-13});
				
				\draw[dotted](5/6,0)node[below]{\small $\frac{5}{6}$}--(5/6,1);
				\draw[dotted](5/7,0)node[below]{\small $\frac{5}{7}$}--(5/7,1);
				\draw[dotted](5/8,0)node[below]{\small $\frac{5}{8}$}--(5/8,1);
				\draw[dotted](5/9,0)node[below]{\small $\frac{5}{9}$}--(5/9,1);
				\draw[dotted](1/2,0)node[below]{\small $\frac{1}{2}$}--(1/2,1);
				\draw[dotted](5/11,0)node[below]{\small $\frac{5}{11}$}--(5/11,1);
				\draw[dotted](5/14,0)--(5/14,1);

		\end{tikzpicture}}

		\caption{  The map $T_N$ for $N=2$ on the left and $N=5$ on the right. }
		\label{fig:T2} 
	\end{figure}
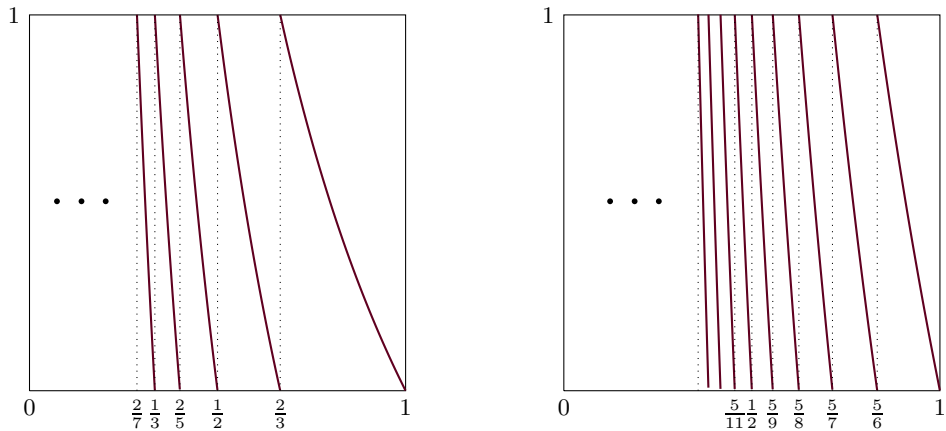

	This continued fraction expansion is finite if and only if $x\in\mathbb{Q}$. We only want to consider expansions with infinitely many digits, hence from here onward we assume $x \in [0,1] \setminus \Q$.  Following the notation of \cite{BGKWY}, we denote $d_n = d_n(x)$, and write $x=[0;d_1,d_2,\ldots]_N$ for the greedy expansion of $x$. Note that this greedy expansion is the unique NCF expansion of $x$ whose digits are all greater than or equal to $N$. On the other hand, each sequence $(d_n)_{n\geq 1}$ with $d_n\geq N$ occurs as a greedy expansion of some irrational $x\in[0,1]\setminus \Q$.
	
	\subsection{{\bf $S$}-adic sequences} 
	Consider a finite alphabet $\cA$ and let $\cA^*$ be the free monoid generated by $\cA$ equipped with the operation of concatenation, that is, $\cA^*$ consists of all the {\it (finite) words} $w_0\cdots w_{n-1}$ with $n\in \N$ and letters $w_0,\ldots, w_{n-1}\in \cA$. The choice $n=0$ corresponds to the empty word which is denoted by $\varepsilon$. A word $u\in\cA^*$ is called a \textit{factor} of $v\in\cA^*$ if $v\in\cA^*u\cA^*$ and we denote it by $u \subset v$. We call $u\in\cA^*$ a \textit{prefix} of $v$ if $v\in u\cA^*$, and a \textit{suffix} of $v$ if $v \in \cA^* u$. We also define $\cA^\N$ as the space of {\it (right) infinite words} or {\it sequences} $w_0w_1\cdots$ with $w_0,w_1,\ldots\in \cA$. We endow $\cA^\N$ with the product topology of the discrete topology on each copy of $\cA$. A word $u\in\cA^*$ is a factor of $\omega\in\cA^\N$ if $\omega\in\cA^*u\cA^\N$ and a prefix of $\omega$ if $\omega\in u\cA^\N$.
	
	For every $u\in\cA^*$, denote by $|u|$ the length ({\it i.e.}, the number of letters) of $u$. For every $a\in\cA$, denote by $|u|_a$ the number of occurrences of the letter $a$ in the word $u$, and, for a word $v$, denote by $|u|_v$ the number of occurrences of $v$ in $u$. Given $\omega \in \cA^* \cup \cA^\N,$ we write $\omega=\omega_0 \omega_1 \omega_2 \cdots$ where $\omega_i \in \cA.$ Given $a \in \cA$ and $d \in \N$, we write $a^d = \underbrace{a\cdots a}_{d\mbox{ times}}$.
	
	We define the \textit{abelianization map} as 
	\begin{equation}\label{abelianizationmap}
		{\bf l}: \cA^*\rightarrow \N^{|\cA|};\quad u\mapsto \,^t\!(|u|_a)_{a \in \cA}.
	\end{equation}
	The \textit{language} $\cL_{\omega}$ of a sequence $\omega$ is given by
	\[\cL_{\omega}:=\{u\in\cA^*\,:\,u\mbox{ is a factor of }\omega \}.	\]
	
	A map $\sigma:\cA\rightarrow\cA^*\setminus\{\varepsilon\}$ is a \textit{substitution} over the alphabet $\cA$. The domain of $\sigma$ can be extended to $\cA^*$ by concatenating the images of each letter, that is, $\sigma$ is an endomorphism over the free monoid $\cA^*$. This allows even to naturally extend the domain of $\sigma$ to the set of sequences~$\cA^\N$.
	
	Given a substitution $\sigma$, we define its \textit{incidence matrix} as the square matrix $M_\sigma=(|\sigma(j)|_i)_{i,j\in\cA}$. This definition immediately implies that ${\bf l}(\sigma(u))=M_\sigma{\bf l}(u)$ for every $u\in\cA^*$. We say that $\sigma$ is unimodular if $\det M_\sigma=\pm1$. The substitutions considered in this article will be nonunimodular in the case $N\ge 2$. 
	
	\begin{definition}[Directive sequence and $S$-adic sequence]\label{def:sadic}	
		Let $\bsigma=(\sigma_n)_{n\geq 1}$ be a sequence of substitutions $\sigma_n:\cA^*\rightarrow\cA^*$ over the alphabet $\cA$. We denote the set of substitutions as $\cS=\{\sigma_n\,:\,n\geq 1 \}$; this set may be finite or infinite. We say that $\bsigma$ is a {\it directive sequence}.
		
		A sequence $\omega\in\cA^\N$ is an {\it $S$-adic sequence} (or  {\it limit sequence}) of the directive sequence $\bsigma=(\sigma_n)_{n\geq 1}$ if there exist $\omega^{(1)}, \omega^{(2)},\ldots\in\cA^\N$ such that
		\begin{equation}\label{eq:limworddef}
		\omega^{(1)}=\omega,\quad \omega^{(n)}=\sigma_{n}(\omega^{(n+1)})\quad\mbox{for all } n \geq 1.
		\end{equation}
	\end{definition}
	
	The best known examples of $S$-adic sequences are furnished by Sturmian sequences. They can be obtained as the (unique) limit sequence of $(\sigma_n)_{n\geq 1}$ where $\sigma_n(0) =0^{d_n} 1$ and $\sigma_n(1) =0$. Sturmian words are related to the classical continued fraction expansion of irrationals. See for instance \cite{A2,AF:01,BST2,BST1,BD,T} for results on $S$-adic sequences.	
	
	\begin{definition}[Generalized right eigenvector]\label{righteigenvector}
		Denote by $\R^d_+$ the set of vectors with positive entries in $\R^d$. Let $(M_n)_{n\geq 1}$ be a sequence of matrices in $\N^{d\times d}$.  A vector $\bu\in\R^d_+$ with $\|\bu\|_1=1$ is said to be a \textit{generalized right eigenvector} for $(M_n)_{n\geq 1}$ if
		\begin{equation}\label{weakconvergence}
			\bigcap_{n\geq 1}M_{1}\cdots M_{n}\R^d_+=\R_+\bu. 
		\end{equation}
		Given a directive sequence $\bsigma=(\sigma_n)_{n\geq 1}$ on an alphabet $\cA$, we say that $\bu\in\R^{|\cA|}_+$ is a generalized right eigenvector of $\bsigma$ if $\bu$ is a generalized right eigenvector of the corresponding sequence of incidence matrices $(M_{\sigma_n})_{n\geq 1}$.
	\end{definition}
	
	In the usual substitutive case (that is, when all substitutions are the same), the incidence matrices all equal a matrix $M$, and one finds that $\bu$ is the Perron-Frobenius eigenvector of $M$, {\it i.e.}, the eigenvector corresponding to the unique largest real eigenvalue (having positive entries).

	On the topological space $\cA^\N$ we consider the left shift $\Sigma$ defined by $\Sigma(\omega_0\omega_1\cdots)=\omega_1\omega_2\cdots$, where $\omega_j\in\cA$. Given a sequence $\omega$ on the alphabet $\cA$, consider the closed set 
	\[  X_\omega := \overline{ \{ \Sigma^n(\omega)\,:\,n\in\N \} }. \]
	Then $(X_\omega,\Sigma)$ constitutes a topological dynamical system called a {\it subshift}. We are interested in ergodic properties of this type of dynamical systems for NCF sequences, which translate to the existence of word frequencies.
	
	\section{$N$-continued fraction sequences}\label{sec:NCFwords}	
	In this section, we will define our main objects of study, which are two families of binary sequences called NCF sequences and dual NCF sequences. To do this, for each irrational $x\in[0,1]$ we will first construct two $S$-adic sequences. The choice of these sequences is what is known as a {\it substitution selection} for the NCF algorithm and its natural extension in the sense of \cite{BST2}. We will justify this in more detail at the end of this section.

	\subsection{Definition of $\boldsymbol{N}$-continued fraction sequences} 	
	We start with relating directive sequences of substitutions to $N$-continued fraction expansions.
	
	\begin{definition}[Directive sequences for $N$-continued fraction expansions]\label{def:subst}
		Let $N \geq 1$ and let $x=[0;d_1,d_2,\ldots]_N \in [0,1] \setminus \Q$.
		\begin{enumerate}
			\item For each $n\geq 1$, consider the substitutions
			\[ 
			\sigma_n:
			\begin{cases}
				0\rightarrow 0^{d_n} 1^N,\\
				1\rightarrow 0.\\
			\end{cases}
			\]
			We assign to  $x$ the directive sequence $\bsigma_x=(\sigma_n)_{n\geq 1}$. We denote $\cS=\{\sigma_n\,:\,n\geq 1 \}$.
			
			\item For each $n\geq 1$, consider the dual substitutions
			\[ 
			\sighat_n:
			\begin{cases}
				0\rightarrow 0^{d_n} 1,\\
				1\rightarrow 0^N.\\
			\end{cases}
			\]
			We assign to $x$ the directive sequence ${\widehat{\bsigma}}_x=(\sighat_n)_{n\geq 1}$. We denote $\widehat{\cS}=\{\sighat_n\,:\,n\geq 1 \}$.
		\end{enumerate}
	\end{definition}
	
	For each $n\geq 1$, the corresponding incidence matrices of $\sigma_n$ and $\sighat_n$ are given by
	$$M_{\sigma_n} = \left(
	\begin{matrix}
		d_n & 1 \\
		N & 0 \\
	\end{matrix}
	\right),\qquad M_{\sighat_n} = \left(
	\begin{matrix}
		d_n & N \\
		1 & 0\\
	\end{matrix}
	\right).$$ 
	Note that they are the transpose of each other. Moreover, easy calculation shows that when $d_n\geq N$, which is the case in the greedy algorithm that we have chosen, the matrices are {\it Pisot}. Recall that a matrix is said to be Pisot if one of its eigenvalues is a real number greater than $1$, and the rest of its eigenvalues have modulus less than $1$. Pisot matrices are very relevant in the study of substitutions (see for instance~\cite{ABBLS}). We would also like to remark that these are not the matrices of the M\"obius transformations associated with the inverse branches of $T_N(x)$. The difference is that the numbers on the diagonal have to be swapped as well as on the anti-diagonal. One can achieve this by relabelling $0$ as $1$ and vice versa. However, this will not affect our results. We chose the substitutions as it is so that the dual substitutions $\sighat_n$ are a particular instance of so called $\beta$-substitutions for simple Parry numbers (see \cite[Section 3.2]{BS} and \cite{FMP}).
	
	Next, we introduce $N$-continued fraction sequences as limit sequences for $\bsigma_x$ and ${\widehat{\bsigma}}_x$. In general, there can be several $S$-adic sequences for the same directive sequence, however in the following definition we refer to ``the" $S$-adic sequence and afterwards we show that it is indeed unique.
	
	\begin{definition}[$N$-continued fraction sequence and its dual]\label{def:omega}
		Let $N \geq 1$ and $x=[0;d_1,d_2,\ldots]_N \in [0,1] \setminus \Q$.
		\begin{enumerate}
			\item We define the \textit{NCF sequence} $\omega(x,N)$ as the $S$-adic sequence of the directive sequence $\bsigma_x=(\sigma_n)_{n\geq 1}$. 
			
			
			\item	We define the \textit{dual NCF sequence} $\ompr(x,N)$ as the $S$-adic sequence of the directive sequence ${\widehat{\bsigma}}_x=(\sighat_n)_{n\geq 1}$. 
		\end{enumerate}
	\end{definition}
	
	\begin{proposition}
	Let $N \geq 1$ and $x=[0;d_1,d_2,\ldots]_N \in [0,1] \setminus \Q$.
	    The NCF sequence $\omega(x,N)$ and its dual $\ompr(x,N)$ are well defined. The finite words $(\sigma_1 \circ \sigma_2 \circ \cdots \circ\sigma_n(1))_{n\geq 1}$ form a nested sequence of prefixes of $\omega(x,N)$ and the finite words $(\sighat_1 \circ \sighat_2  \circ \cdots \circ\sighat_n(0))_{n\geq 1}$ form a nested sequence of prefixes of $\ompr(x,N)$.
	\end{proposition}
	
	\begin{proof}
	    Consider the directive sequence $\bsigma_x=(\sigma_n)_{n\geq 1}$, and note that $\sigma_n(a)$ starts with $0$ for $a \in \{0,1\}$ and for every $n \geq 1$. This implies that $\sigma_1 \circ \sigma_2 \circ \cdots \circ\sigma_n(1)$ is a prefix of $\sigma_1 \circ \sigma_2 \circ \cdots \circ\sigma_{n+1}(1)$ for each $n\in \N$. Moreover, by the definition of $(\sigma_n)_{n\in\N}$, the length of the words $\sigma_1 \circ \sigma_2 \circ \cdots \circ\sigma_n(1)$ tends to infinity for $n\to\infty$ (for a detailed study of this we refer to Section~\ref{section:letterfrequency}). Suppose that $\nu$ is a limit sequence of $\bsigma_x$. By the definition of a limit sequence in \eqref{eq:limworddef} there exist $\nu^{(1)}, \nu^{(2)}, \ldots \in \{0,1\}^\N$ such that $\nu=\nu^{(1)} = \sigma_1\circ \cdots \circ \sigma_{n}(\nu^{(n+1)})$ for each $n$. Because $\sigma_{n}(\nu^{(n+1)})$ starts with $0=\sigma_n(1)$ this implies that $\nu$ starts with $\sigma_1 \circ \sigma_2  \circ \cdots \circ\sigma_{n-1}(0)=\sigma_1 \circ \sigma_2  \circ \cdots \circ\sigma_{n}(1)$. Thus the finite words $(\sigma_1 \circ \sigma_2  \circ \cdots \circ\sigma_n(1))_{n\geq 1}$ form a nested sequence of prefixes of $\nu$ and, hence, 
	      \[
			\nu = \lim_{n \rightarrow \infty} \sigma_1 \circ \sigma_2  \circ \cdots \circ\sigma_n(111\cdots),
		\]
where the limit is taken with respect to the topology we defined on $\cA^\N$. Therefore, there is exactly one $S$-adic sequence $\nu$ for the directive sequence $\bsigma_x$, namely the sequence $\nu=\omega(x,N)$, and the finite words $(\sigma_1 \circ \sigma_2 \circ \cdots \circ\sigma_n(1))_{n\geq 1}$ form a nested sequence of prefixes of $\omega(x,N)$.
	    
The assertions about $\ompr(x,N)$ follows along the same lines.
%
	\end{proof}

    We take a look at an example: let $N=2$ and $x= [0;2,3,4,5,\ldots]_2$, then the corresponding words obtained are
	\[ \omega(x,N) = 0011001100110000110011001100001100110011000011001100110000110011001\cdots
	\]
	and
	\[ \ompr(x,N) = 0010010010000100100100001001001000010010010000100100100100100001001\cdots
	\]
	By Definition~\ref{def:sadic}, if $\omega(x,N)$ is the $S$-adic sequence of the directive sequence $\bsigma_x=(\sigma_n)_{n\geq 1}$, then there exist $\omega^{(1)}, \omega^{(2)},\ldots\in\cA^\N$ such that $\omega^{(1)}=\omega(x,N)$ and $\omega^{(n)}=\sigma_{n}(\omega^{(n+1)})$ for all $n \geq 1.$ It is not hard to see that $\omega^{(n)} =  \omega(T^{n-1}_N(x),N)$ because $T_N(x) = [0; d_2, d_3, \ldots]_N.$ The same holds for $\ompr(x,N).$
	
	
	We introduce the following notation.
	Define the words \[\subw{0}:=1,\quad\subw{n}:= \sigma_1\circ\sigma_2\circ\cdots\circ \sigma_n(1)\quad \mbox{ for }n\geq 1.\]
	By the definition of the substitutions $\sigma_n$  in Definition~\ref{def:subst}~(1) the words $\subw{n}$  satisfy the recurrence 
	\begin{equation}\label{recurrencesigma}
		\subw{n+1}=\subw{n}^{d_n}\subw{n-1}^N \quad \mbox{ for }n\geq 1.
	\end{equation}
	The first iterations are
	\begin{equation*}
		\begin{split}
			\subw{0}&=1,\quad \subw{1}=0,\quad\subw{2}=0^{d_1}1^N,\quad \subw{3}=\underbrace{0^{d_1}1^N0^{d_1}1^N\cdots 0^{d_1}1^N}_{d_2 \mbox{ times}}0^N,\\
			\subw{4}&=\underbrace{\underbrace{0^{d_1}1^N0^{d_1}1^N\cdots 0^{d_1}1^N}_{d_2 \mbox{ times}}0^N\cdots \underbrace{0^{d_1}1^N0^{d_1}1^N\cdots 0^{d_1}1^N}_{d_2 \mbox{ times}}0^N}_{d_3 \mbox{ times}} \underbrace{ 0^{d_1}1^N \cdots 0^{d_1}1^N}_{N \mbox { times }}.
		\end{split}
	\end{equation*}
	Analogously, let \[\sipr_0:=1,\quad \sipr_1:=0,\quad  \sipr_{n+1}:=\sighat_1 \circ \sighat_2\circ \cdots \circ\sighat_n(0)\quad \mbox{ for }n\geq 1.\] 
	Then from Definition~\ref{def:subst}~(2) we immediately see that
	\begin{equation}\label{recurrencesighat}
		\sipr_{n+1}=\sipr_n^{d_n}\sipr_{n-1}^N \quad\mbox{ for } n\geq 2.
	\end{equation}
	The first iterations are
	\begin{equation*}
		\begin{split}
			\sipr_{0}&=1,\quad \sipr_{1}=0,\quad\sipr_{2}=0^{d_1}1,\quad \sipr_{3}=\underbrace{0^{d_1}10^{d_1}1\cdots 0^{d_1}1}_{d_2 \mbox{ times}}0^N,\\
			\sipr_{4}&=\underbrace{\underbrace{0^{d_1}10^{d_1}1\cdots 0^{d_1}1}_{d_2 \mbox{ times}}0^N\cdots \underbrace{0^{d_1}10^{d_1}1\cdots 0^{d_1}1}_{d_2 \mbox{ times}}0^N}_{d_3 \mbox{ times}} \underbrace{ 0^{d_1}1 \cdots 0^{d_1}1}_{N \mbox { times }}.
		\end{split}
	\end{equation*}
	
	Let $\omega\in\{0,1\}^\N$ be a sequence over two letters and $a,b\in\{0,1\}$ with $a\not=b$. We say that a factor $v=a\cdots a$ of $\omega$ is a \textit{maximal $a$ block} of $\omega$ if either $vb$ is a prefix of $\omega$ or $bvb$ is a factor of $\omega$. The recurrences \eqref{recurrencesigma} and \eqref{recurrencesighat} immediately allow us to characterize maximal $a$ blocks for $\omega(x,N)$ and $\ompr(x,N)$ according to the following lemma.
	
	\begin{lemma}\label{lem:maxFactor} 
		Let $N\ge 1$ and $x=[0;d_1,d_2,\ldots]_N$ be given. 
		\begin{enumerate}
			\item[(1a)] $1^N$ is the only maximal $1$ block of $\omega(x,N)$.
			\item[(1b)] $0^{d_1}$ and $0^{d_1+N}$  are the only maximal $0$ blocks of $\omega(x,N)$.
			\item[(2a)] $1$ is the only maximal $1$ block of $\ompr(x,N)$.
			\item[(2b)] $0^{d_1}$ and $0^{d_1+N}$  are the only maximal $0$ blocks of $\ompr(x,N)$.
		\end{enumerate}
	\end{lemma}
	
	Note that the NCF sequence $\omega(x,N)$ can be obtained from the dual NCF sequence $\ompr(x,N)$ by substituting each occurrence of $1$ in $\ompr(x,N)$ by $1^N$. This is true because the sequences $(\subw{n})_{n\geq 1}$ and $(\sipr_n)_{n\geq 1}$ satisfy the same recurrence formula, and the only difference is that $\subw{2}=0^{d_1}1^N$ and $\sipr_2=0^{d_1}1$.
	
	Formally, consider the substitution
	\begin{equation}\label{eq:sdef}
		\tau:
		\begin{cases}
			0 \mapsto 0,\\
			1 \mapsto 1^N.
		\end{cases}
	\end{equation}
	Then 
	\begin{equation}\label{omegaSomega}
		\omega(x,N) = \tau(\ompr(x,N))
	\end{equation}
	holds. Hence, the dual NCF has essentially the ``same shape'' as the regular one but is a bit easier to work with. We will later use this correspondence to be able to transfer properties of one sequence to the other. The sequence $\omega(x,N)$ has the advantage that the ratio of the letter frequencies converges to $x$, which makes it a more natural choice.

	\subsection{Letter frequency and generalized right eigenvector}\label{section:letterfrequency}
	
	Let $N \geq 1$ and consider the expansion $x=[0;d_1,d_2,\ldots]_N \in [0,1] \setminus \Q$. Define the convergents $c_n=\frac{p_n}{q_n}$ for $n \geq 1$ as 
	\[ \frac{p_n}{q_n}:= [0;d_1,d_2,\ldots,d_n]_N, \] 
	and choose $p_n$ and $q_n$ so that they satisfy the following recurrence relations:
	\begin{equation}\label{convergents}
	    \begin{split}
		    p_{-1}=1, \quad p_0=0, \quad p_n=d_np_{n-1}+Np_{n-2},\\
	        q_{-1}=0, \quad q_0=1, \quad q_n=d_nq_{n-1}+Nq_{n-2}.
		\end{split}
	\end{equation}
	
	Then we have $x=\lim_{n\to\infty}\frac{p_n}{q_n}$. Note that for the classical case $N=1$, $p_n$ and $q_n$ are coprime for all $n$. For $N\geq 2$ this is not necessarily the case, but still $c_n=\frac{p_n}{q_n}$ for $n \geq 1$. This can be shown following the same lines as in the well known setting of the regular continued fraction algorithm.

	Consider the directive sequence $\bsigma_x=(\sigma_n)_{n\geq 1}$ and the corresponding sequence of incidence matrices $(M_{\sigma_n})_{n\geq1}.$ Set $M_{[1,n]}=M_{\sigma_1}M_{\sigma_2}\cdots M_{\sigma_n}$. 
	Then we find that 
	\[
	M_{[1,n]}=\left(
	\begin{matrix}
		q_n & q_{n-1} \\
		p_n & p_{n-1}
	\end{matrix}
	\right).
	\]
	Since $\det(M_{[1,n]})=\det(M_{\sigma_1})\det(M_{\sigma_2})\cdots \det(M_{\sigma_n})=(-N)^n$, we have $q_np_{n-1}-p_nq_{n-1}=(-N)^n$ which we will use in Section \ref{sec:dynamics}. Furthermore,	since ${\bf l}(1)=\, ^t\!(0,1)$, we have
	\[ {\bf l}(\sigma_1 \circ  \cdots \circ\sigma_n(1))= M_{[1,n]}  \left(
	\begin{matrix}
		0 \\
		1
	\end{matrix}
	\right)= \left(
	\begin{matrix}
		q_{n-1} \\
		p_{n-1}
	\end{matrix}
	\right) \]
	and, hence, we gain
	\begin{equation}\label{eq:letterfreq}
		\lim_{n\to\infty} \frac{|\sigma_1 \circ \cdots \circ\sigma_n(1)|_1}{|\sigma_1 \circ \cdots \circ\sigma_n(1)|_0} =\lim_{n\to\infty} \frac{p_{n-1}}{q_{n-1}}=x.
	\end{equation}
	
	We define the {\it frequency } of a letter $a\in\cA$ in the sequence $\omega\in\cA^\N$ as
	\[ f_a := \lim_{|p|\to\infty} \frac{|p|_a}{|p|}, \]
	provided that the limit, which is taken over the prefixes $p$ of $\omega$, exists. If the limit does not exist, we say that $a$ does not have a frequency in $\omega$.  
	Equation \eqref{eq:letterfreq} implies that the $N$-continued fraction sequence $\omega(x,N)$ has letter frequencies and the  {\it frequency vector} is given by $(f_0,f_1)=\left(\frac{1}{x+1},\frac{x}{x+1}\right).$ We show next that this vector is in fact a generalized right eigenvector for~$\bsigma_x.$ We will use the following auxiliary lemma.
	
	\begin{lemma}[Birkhoff~\cite{Bir:57}] \label{lem:355}
		Let $(B_n)_{n\geq 1}$ be a sequence of matrices with nonnegative entries. If there exists a matrix $B$ with strictly positive entries, an integer $h>0$, and a strictly increasing sequence $(m_i)_{i\geq 1}$ of positive integers such that $B=B_{m_i}\cdots B_{m_i+h}$ for each $i \geq 1$, then $(B_n)_{n\geq 1}$ has a generalized right eigenvector.
	\end{lemma}
	
    See also Furstenberg~\cite{Furstenberg:60} and for instance the proof of \cite[Proposition~3.5.5]{T} where a  slightly weaker statement is given.	Lemma \ref{lem:355} is used in the proof of the following result.	
	
	\begin{lemma}\label{righteigen}
		The directive sequence $\bsigma_x=(\sigma_n)_{n\geq 1}$ of NCF substitutions has a generalized right eigenvector given by $(f_0,f_1)=\left(\frac{1}{x+1},\frac{x}{x+1}\right).$ 
	\end{lemma}
	
	\begin{proof}
		Let $$A=\left(\begin{matrix}
			1 & \frac{1}{N}\\
			0 & 1
		\end{matrix}\right)\quad \text{and}  \quad 
		D=\left(
		\begin{matrix}
			N & 1 \\
			N & 0 
		\end{matrix}
		\right). $$
		Note that, for every $n\geq 1$, 
		\[
		M_{\sigma_n}= \begin{pmatrix}
			d_n & 1 \\
			N & 0 
		\end{pmatrix} = \begin{pmatrix}
			1 & \frac{d_n - N}{N}\\
			0 & 1
		\end{pmatrix}\cdot \begin{pmatrix}
			N & 1 \\
			N & 0
		\end{pmatrix}= A^{d_n-N}D. 
		\]
		
		Consider the sequence of matrices 
		
		\[
		(M'_n)_{n\geq 1}=(\underbrace{A,\ldots,A}_{d_1-N \text{ times}},D,\underbrace{A,\ldots,A}_{d_2-N \text{ times}},D,\ldots),
		\]
		
		which satisfies $\prod_{k=1}^{d_1-N+1} M'_k = M_{\sigma_1},$ and the product of the next block of $d_2-N+1$ gives $M_{\sigma_2}$, etc. Therefore we find $\bigcap_{n\geq 1}M_{\sigma_n} \R^2_+=\bigcap_{n\geq 1} M'_n \R^2_+$. Thus $\bsigma_x=(\sigma_n)_{n\geq 1}$ has a generalized right eigenvector if and only if $(M'_n)_{n\geq 1}$ has one. 
		
		If the sequence of NCF digits $(d_n)_{n\geq 1}$ is eventually equal to $N$, then $(M'_n)_{n\geq 1}$ is eventually equal to $D$. Because $D^2$ has only positive entries, the conditions of Lemma~\ref{lem:355} are satisfied and $(M'_n)_{n\geq 1}$ has a generalized right eigenvector. Otherwise, the sequence $(M'_n)_{n\geq 1}$ changes infinitely many times between $A$ and $D$, hence there exists a strictly increasing sequence of integers $(m_i)_{i\geq 1}$ such that $M'_{m_i}M'_{m_i+1}=DA$, which has strictly positive entries. Thus the existence of a generalized right eigenvector follows again from Lemma~\ref{lem:355}. 
		
		We conclude that $(M'_n)_{n\geq 1}$, and hence $\bsigma_x$, have a generalized right eigenvector. It follows from the proof of \cite[Lemma 3.5.10]{T} that, whenever $\bsigma_x$ has a generalized right eigenvector, its entries correspond to that of the letter frequency vector of the limit sequence. This finishes the proof.
	\end{proof}
	
	\subsection{Substitution selection for the $N$-continued fraction algorithm}\label{sec:subselection}
	
	We now want to relate the sequences of substitutions $\bsigma_x$ and ${\widehat{\bsigma}}_x$ to the NCF algorithm. It turns out that these sequences of substitutions can be regarded as a combinatorial version of the NCF algorithm and its dual. For Sturmian sequences and their directive sequences, their multi-faceted interplay with the classical continued fraction algorithm is well-known (see {\it e.g.}\ Arnoux and Rauzy~\cite{Arnoux-Rauzy:91} or Arnoux and Fisher~\cite{AF:01}). Berth\'e {\it et al.}~\cite{BST2} generalized this to unimodular multidimensional continued fraction algorithms by introducing the concept of  {\it substitution selection}. The novelty of our setting is that we are working with nonunimodular matrices.

	Let $N\geq 1$. We want to define a version of the NCF algorithm that works on the subset $\mathbb{P}_<=\{[1:x]\,:\, 0<x<1 \}$ of the projective space $\mathbb{P}$.  Let ${\bf x}=[1:x]\in\mathbb{P}_<$ and consider the matrix 
	$$C_N({\bf x})=
	\left(
	\begin{matrix}
		\left\lfloor \frac{N}{x} \right\rfloor & N \\
		1 & 0 \\
	\end{matrix}
	\right).$$
	Then the map
	\[ 
	G_N\,:\,  \mathbb{P}_< \rightarrow \mathbb{P}_<, \quad
	{\bf x} \mapsto {^t}C_N({\bf x})^{-1} {\bf x}
	\]
	is called the {\it linear multiplicative $N$-continued fraction algorithm}. It has the form
	\[ 
	G_N([1:x])=
	\begin{pmatrix}
		0 & \frac1N \\ 1 & -\frac1N \left\lfloor \tfrac{N}{x} \right\rfloor
	\end{pmatrix}\cdot [1:x]
	= \left[ \frac xN : 1-\frac xN  \left\lfloor \frac{N}{x} \right\rfloor \right]=
	\left[ 1 : \frac{N}{x}-\left\lfloor \frac{N}{x} \right\rfloor  \right],
	\]
	which is a projectivization of $T_N$ in the sense that the original mapping $T_N$ can be seen in the second coordinate of $G_N$ if the first coordinate is normalized to $1$. If $x=[0;d_1,d_2,\ldots]_N \notin \Q $ then we have for each $n\geq 1$ that
	\begin{equation}\label{eq:CNandM}
	 {^t}C_N ( G_N^{n-1} ( {\bf x} ) ) = {^t}C_N([ 1 : T_N^{n-1}(x) ]) = \left(
	\begin{matrix}
		\left\lfloor \frac{N}{T_N^{n-1}(x)} \right\rfloor & 1 \\
		N & 0 \\
	\end{matrix}
	\right) =  \left(
	\begin{matrix}
		d_n & 1 \\
		N & 0\\
	\end{matrix}
	\right)=M_{\sigma_n}.
	\end{equation}
	Iteration yields 
	\[
	\begin{split}
		\bx &= \, ^t\! C_N({\bf x})  G_N( \bx) =  \, ^t\! C_N({\bf x}) \, ^t\! C_N(G_N({\bf x}))  G_N^2 (\bx)
		=\dots = \, ^t\! C_N({\bf x})\cdots  \, ^t\! C_N(G_N^{n-1}({\bf x}))  G_N^{n} (\bx) 
	\end{split}
	\]
	and therefore
	\[
	\begin{split}
		\bx = M_{\sigma_1} G_N( \bx) = M_{\sigma_1} M_{\sigma_2} G_N^2 (\bx)
		=\dots = M_{\sigma_1}\cdots M_{\sigma_n} G_N^{n} (\bx).
	\end{split}
	\]
	Thus the NCF algorithm applied to $x=[0;d_1,d_2,\ldots]_N$ produces the incidence matrices of the substitutions $\sigma_{1},\sigma_{2},\ldots$ given in Definition~\ref{def:subst}~(1). Moreover, by Lemma~\ref{righteigen}, the ray $[1:x]$ corresponds to the direction of the generalized right eigenvector of these substitutions. In this sense, the directive sequence $\bsigma_x$ can be regarded as a {\it substitution selection} of the NCF algorithm (see the definition of substitution selection in \cite[Definition 2.2]{BST2}).	
	Indeed, given $ \cS=\{\sigma_n\,:\,n\geq 1 \}$, we consider the map
	\[ 
	{\bf \varphi} \,:\, [0,1] \setminus \Q \rightarrow \cS^\N,\quad {x} \mapsto \bsigma_x=(\sigma_n)_{n\geq 1}  
	\]
	and endow the space $\cS^\N$ with the left shift $\Sigma$, that is, $\Sigma((\sigma_n)_{n\geq 1})=(\sigma_{n+1})_{n\geq 1}$. Then the following diagram commutes:
	
	\[
	\xymatrix{
		[0,1] \setminus \Q \ar[r]^{T_N} \ar[d]_{\varphi} & [0,1] \setminus \Q \ar[d]^{\varphi} \\
		\cS^\N \ar[r]^{\Sigma}&\cS^\N
	}
	\]
	Thus we can associate the limit sequence $\omega(x,N)$ of $\bsigma_x$ to the NCF expansion of $x$ in the same way as Sturmian sequences are associated with the classical continued fraction expansion of $x$, {\it e.g.}\ in \cite{A,AF:01,T}.

	As in the classical case (see~\cite{AF:01}) we go one step further and associate the symbolic sequence $\ompr(x,N)$ to the past of the natural extension of $T_N$. The associated directive sequence $\widehat{\bsigma}_x$ is in some sense a dual of $\bsigma_x$.
	Let ${\bf x}=[1:x]$ and ${\bf y}=[1:y]$ be elements of $\mathbb{P}_<$. Following \cite{AL}, a natural extension of the map $G_N$ is given by 
	\begin{equation}\label{eq:GnNatl}
		\widetilde{G}_N:\mathbb{P}_<^2\rightarrow \mathbb{P}_<^2,\quad
		\left(\begin{matrix}
			{\bf x}\\
			{\bf y}
		\end{matrix}
		\right)\mapsto \left(\begin{matrix}
			^t C_N({\bf x})^{-1} & 0 \\
			0 &  C_N({\bf x})
		\end{matrix}
		\right)
		\left(\begin{matrix}
			{\bf x}\\
			{\bf y}
		\end{matrix}
		\right)=\left(\begin{matrix}
			\left[ 1 : \frac{N}{x}-\left\lfloor \frac{N}{x} \right\rfloor\right]\\[1mm]
			\left[ 1 : \frac{1}{N\cdot y+\left\lfloor \tfrac{N}{x} \right\rfloor }\right]
		\end{matrix}
		\right).
	\end{equation}
	Taking second coordinates and inspecting the range of $\mathbf{y}$ this immediately yields the following result.
	
	\begin{proposition}\label{prop:natl}
		A  natural extension of the map $T_N : [0,1]  \rightarrow [0,1] $ is given by 
		\begin{equation}
		    \begin{split}
		        \widetilde{T}_N:&[0,1] \times \left[0,\frac{1}{N}\right]  \rightarrow [0,1] \times  \left[0,\frac{1}{N}\right],\\
		            &(x,y)\mapsto\begin{cases}
		                \left( T_N(x), \frac{1}{N\cdot y+\left\lfloor \tfrac{N}{x} \right\rfloor }\right) & x \neq 0,\\
		                (0,0) & x = 0,
		           \end{cases}
		    \end{split}
		\end{equation}
		with $\frac{dx dy}{(1+xy)^2}$ as invariant measure. 
	\end{proposition}	
	
	We mention that in \cite{DKW}, a different natural extension of $T_N$, which is isomorphic to ours, is given (the difference is that $y$ is replaced by $\frac{y}{N}$). As we see from \eqref{eq:GnNatl}, the ``past'' of this natural extension (that is, the second coordinate of $ \widetilde{T}_N$) is associated with $C_N({\bf x})$. Moreover, $C_N ( G_N^{n-1} ( {\bf x} ) ) = M_{\sighat_n}$ for every $n \geq 1$ (this is straightforward from \eqref{eq:CNandM}). Therefore, the natural extension of the NCF algorithm is related to the incidence matrices of the dual substitutions $ \sighat_n$, and we can associate it with directive sequences of the form ${\widehat{\bsigma}}_x=(\sighat_n)_{n\geq 1}$ and, hence, with the $S$-adic words $\ompr(x,N)$. 
	
	These ``substitution selections'' superimpose a combinatorial structure on the NCF algorithm and its natural extension. As is detailed in \cite{A,AF:01,AF:05,T}, this combinatorial structure gives information about the underlying continued fraction algorithm and vice versa. This motivates our study of the $S$-adic sequences $\omega(x,N)$ and $\ompr(x,N)$.
	
	\section{Balance Properties of $N$-continued fraction sequences}\label{sec:balance}
	
	In this section, we prove that the sequences $\omega(x,N)$ and $\ompr(x,N)$ are finitely balanced. For each fixed $N\ge 1$ we provide an upper bound for the balance constant that is valid for each~$x$. We refine this result by defining some sets in terms of the size of the NCF digits for which the balance constant can be improved. After that, we provide lower bounds for the balance constant. We refer the reader to \cite{A} for some notions on balance of sequences.	
	
	\subsection{Balance results for $N$-continued fraction sequences}	
	We begin with the definition of balance.
	
	\begin{definition}[Balance]
		Given $C>0$, we say that a pair $(u,v)$ of words over the alphabet $\cA$ is {\it $C$-balanced} if $|u|=|v|$ and
		\[-C\leq |u|_a-|v|_a\leq C\quad \mbox{ for every }a\in\cA. \] 
		We say that a sequence $\nu \in \cA^\N$ is \textit{$C$-balanced} if every pair $(u,v)$ of factors of $\nu$ with $|u|=|v|$ is $C$-balanced.
		We say that $\nu$ is {\it finitely balanced} if it is $C$-balanced for some $C>0$.
	\end{definition}

	Before we prove finite balancedness for NCF sequences and their duals, we introduce the following definition of a minimal pair. 
	
	\begin{definition}[Minimal pair]
		Let $\nu\in \cA^\N$, two factors $u,v \subset \nu$ and $C>0$. We say that $(u,v)$ is a \textit{minimal pair} of not $C$-balanced factors if $|u|=|v|$, $||v|_a - |u|_a| > C$ for some $a \in \cA$, and the length of $u$ and $v$ is minimal with respect to this property among all factors of $\nu$.
	\end{definition}
	
	We have mentioned that it is easier to work with the dual sequence $\ompr(x,N)$ than with the related sequence $\omega(x,N)$. Thus our strategy is to prove results for dual NCF sequences and translate them to NCF sequences afterwards. The following lemma allows us to transfer the property of balancedness in this way.
	
	\begin{lemma}\label{lem:NCbal}
			If $\ompr(x,N)$ is $C$-balanced for some $C>0$, then $\omega(x,N)$ is $N\cdot C$-balanced.
	\end{lemma}
	
	\begin{proof}
		Suppose this is not true, then  given that $\ompr(x,N)$ is $C$-balanced for some $C>0$ there exists a minimal pair $(u,v)$ of not $N\cdot C$-balanced factors of $\omega(x,N)$. It is clear that $u$ and $v$ cannot start with the same letter nor end with the same letter. Assume w.l.o.g. that $|v|_1 > |u|_1$, then $v$ must start and end with $1$ and $u$ must start and end with $0$.	Thus according to Lemma~\ref{lem:maxFactor} we can write $v=1^j 0^{d_1} V 0^{d_1} 1^k$ for some word $V$ and  $ j,k\in\{1,\ldots, N\}$. We distinguish two cases.
		
		If $j + k \leq N$, then Lemma~\ref{lem:maxFactor} implies that $\widetilde{v} = 0^{d_1+j+k-N} V 0^{d_1} 1^N$ is also in $\omega(x,N)$ and satisfies $|\widetilde{v}|_1\ge |v|_1$ and $|\widetilde{v}|=|v|$. Then because $(u,\widetilde{v})$ is not $N\cdot C$-balanced and both words start with $0$, $(u,v)$ is not a minimal pair of not $N\cdot C$-balanced factors, a contradiction. 
		
		Suppose now that $N + 1 \leq j + k \leq 2N.$  It is clear that, if $(u,v)$ is a minimal pair of not $N\cdot C$-balanced factors, then $|v|_1 - |u|_1 = N\cdot C + 1$. Because $u$ starts and ends with $0$, $|u|_1$ is a multiple of $N$, so we must have $j + k = N + 1$. Hence, we can assume w.l.o.g. that $v=1^N 0^{d_1} V 0^{d_1} 1$, because according to Lemma~\ref{lem:maxFactor} we can adjust the values of $j$ and $k$ by ``shifting" $v$. Consider the substitution $\tau$ defined in \eqref{eq:sdef}. Then by \eqref{omegaSomega} we have that $\omega(x,N) = \tau(\ompr(x,N))$. Find $\widehat{u},\widehat{v}\subset \ompr(x,N)$ such that $\tau(\widehat{u})=u$ and $\tau(\widehat{v})=v1^{N-1}$. Then $|\widehat{v}|_1>|\widehat{u}|_1$ and $|\widehat{v}|<|\widehat{u}|$. Find a word $\widehat{V}$ such that $\widehat{v}\widehat{V}$ is a factor of $\ompr(x,N)$ and $|\widehat{v}\widehat{V}|=|\widehat{u}|$. Then, because $\ompr(x,N)$ is $C$-balanced by hypothesis, $|\widehat{v}\widehat{V}|_1-|\widehat{u}|_1 \leq C$. Therefore, by the definition of $\tau$, $|\tau(\widehat{v}\widehat{V})|_1 - |\tau(\widehat{u})|_1 \leq N\cdot C. $ But $|v|_1 \leq |\tau(\widehat{v}\widehat{V})|_1$ and $\tau(\widehat{u})=u$, which implies $0<|v|_1-|u|_1\leq N\cdot C$. This is again a contradiction to the assumption that $(u,v)$ is a minimal pair of not $N\cdot C$-balanced factors of $\omega(x,N)$.
	\end{proof}

	The following theorem states the existence of balance constants for NCF sequences and their duals, which depend on $N$ and on a lower bound for the NCF digits of $x$. We mention that this generalizes \cite[Theorem 4.1]{Tu}, where the result is proven for the case where all the NCF digits are the same.

	\begin{theorem}\label{prop:Nbalalg}
		Let $N\geq1$ be fixed and set $K\geq N$ and $
		C=
		\big\lfloor \frac{K-1}{K+1-N} \big\rfloor  +1 
		$. If we set  $$W_{K,N} := \{ [0;d_1,d_2,\ldots ]_N  \in [0,1]\setminus \Q \,:\, d_n \geq K \text{ for all  } n \geq 1 \}$$ then the following assertions hold. 
		\begin{enumerate}
			\item For all $x \in W_{K,N} $ the dual NCF sequence $\ompr(x,N)$ is $ C $-balanced.
			\item For all $x \in W_{K,N} $ the NCF sequence $\omega(x,N)$ is $ N \cdot C $-balanced.
		\end{enumerate}
	\end{theorem}
	
	\begin{proof}
		We start with the proof of (1). Because the result is well-known for $N=1$ we may assume that $N\ge 2$. Suppose this assertion is not true for some fixed $K$ and $N$ with $K\ge N\ge 2$. Then there exists $x \in W_{K,N}$, such that the sequence $\ompr = \ompr(x,N)$ admits a minimal pair $(u,v)$ of not $C$-balanced factors. We assume that $x$ is chosen in a way that $(u,v)$ has minimal length among all not $C$-balanced pairs of factors of $\ompr(y,N)$ with $y \in W_{K,N}$. We will reach a contradiction by finding a not $C$-balanced pair of shorter length.
		
		Let $d=d_1$. By minimality of $(u,v)$ we have $||v|_1 - |u|_1| = C + 1$. Also, $u$ and $v$ cannot start with the same letter nor end with the same letter. Assume w.l.o.g. that $|v|_1 > |u|_1$, then $v$ must start and end with $1$.	Let $0^s1$ be a prefix of $u$ (it is easy to see that $u$ has to contain an occurrence of $1$ because otherwise $||v|_1 - |u|_1| \le 2 < C+1$). Then $s \geq d+1$, otherwise we could always remove the prefix $0^s1$ from $u$ and the suffix $0^s1$ from $v$ (note that for $s\le d$ the word $0^s1$ has to be a suffix of $v$ by Lemma~\ref{lem:maxFactor}) and find a shorter pair of not $C$-balanced words. Analogously, if $10^t$ is a suffix of $u$ then $t \geq d+1$. Moreover, $s + t \geq 2d + N +1$, otherwise we could ``shift" $u$ and find a word $\widetilde{u}$ with the prefix $0^s1$ with $s \leq d$ such that $(\widetilde{u},v)$ is a minimal pair of not $C$-balanced words, which contradicts what we just proved.
		
		Summing up, we may assume w.l.o.g. $u=0^s1\cdots 10^{d+N}$, $d+1 \leq s \le d+N$, and  $v=10^d\cdots 0^d1$.  For a factor $w$ of $\omega$ denote by $|w|_{0^d*}$  the number of occurrences of a maximal $0$ block of length $0^d$ in $w$, that is, $|w|_{0^d*}:=|1w1|_{10^d1}$. Lemma~\ref{lem:maxFactor} implies that\footnote{Note that we need to use $|\cdot|_{10^{d+N}}$ because otherwise the prefix $0^s$ in $u$ is counted twice if $s=d+N$.}
		\[ |v|_0 = (d+N) \, |v|_{10^{d+N}} + d \,  |v|_{0^d*}
		\]
		and 
		\[ |u|_0 = s + (d+N) \, |u|_{10^{d+N}} + d \, |u|_{0^d*}.\]
		Then we have, 	
		\begin{equation}\label{eq:split}
			\begin{split}
				|u|_0 - |v|_0 &= s + (|u|_{10^{d+N}} - |v|_{10^{d+N}})(d+N) +( |u|_{0^d*} -|v|_{0^d*} )\,d\\
				&= s + (|u|_{10^{d+N}} - |v|_{10^{d+N}})\,N +((|u|_{10^{d+N}} + |u|_{0^d*}) -(|v|_{10^{d+N}} + |v|_{0^d*}) )\,d.
			\end{split}
		\end{equation}

		Every maximal $0$ block of $\ompr$ lies between $1$'s. Since $v$ starts and ends with $1$, the number of maximal $0$ blocks of $v$ is $|v|_1 - 1$, and since $u$ starts and ends with $0$, the number of $0$ blocks in $u$ is $|u|_1 + 1$, of which exactly one of them (the prefix $0^s$ of $u$) is not counted by the terms $|u|_{10^{d+N}}$ and $|u|_{0^d*}$. Now, because $|v|_1 - |u|_1 = C + 1$, this yields
		\begin{equation}\label{eq:reverteduv-1}
			(|v|_{10^{d+N}} + |v|_{0^d*}) - (|u|_{10^{d+N}} + |u|_{0^d*}) = C.
		\end{equation}
		Let \[ a:= |u|_{10^{d+N}} - |v|_{10^{d+N}}.  \] Then from \eqref{eq:split} and \eqref{eq:reverteduv-1} we obtain
		\begin{equation}\label{eq1}
			|u|_0 - |v|_0 = s + a \cdot N - d \cdot C.
		\end{equation}
		Also, $|u|=|v|$ and hence \begin{equation}\label{eq2}
			|u|_0 - |v|_0 = |v|_1 - |u|_1 = C + 1,
		\end{equation} so combining \eqref{eq1} and \eqref{eq2} yields
		\begin{equation}\label{eq3}
			a = \frac{(d+1) \, C + 1 - s}{N}.
		\end{equation}	
		Since $\ompr$ is an $S$-adic sequence, by Definition~\ref{def:sadic} we have $\ompr = \sighat_1(\ompr^{(2)})$ where $\ompr^{(2)}=\ompr(T_N(x) , N)$ with $T_N(x) = [0;d_2,d_3,\ldots]_N \in W_{K,N}$. As a consequence of the shape of $\sighat_1$ there exist words $u^{(2)},v^{(2)} \subset \ompr^{(2)}$ such that
		\[ \sighat_1(u^{(2)}) = 0^{d + N - s} u 1, \qquad \sighat_1(v^{(2)}) =  0^d v.	\]
		We claim that
		\begin{equation}\label{eq:u1v1pos}
			|v^{(2)}| \le |u^{(2)}|.
		\end{equation}
		To prove this, suppose on the contrary that  $|v^{(2)}| - |u^{(2)}| \geq 1$. Note that the letter $1$ appears in $\sighat_1(u^{(2)})$ (resp. $\sighat_1(v^{(2)})$) whenever a $0$ appears in $u^{(2)}$ (resp. $v^{(2)}$). Hence, \eqref{eq2} yields 
		\begin{equation}\label{eq4}
			|v^{(2)}|_0 - |u^{(2)}|_0 = |0^dv|_1 - |0^{d+N-s}u1|_1 = C.
		\end{equation} 
		Hence, $$|v^{(2)}|_1 -  |u^{(2)}|_1 = |v^{(2)}| - |u^{(2)}| + |u^{(2)}|_0 - |v^{(2)}|_0 \geq 1 - C.$$ By the definition of $\sighat_1$ and using that $C \geq 1$ this implies that 
		\begin{equation}
			\begin{split}
				|\sighat_1(v^{(2)})| - |\sighat_1(u^{(2)})| & = (d+1)(|v^{(2)}|_0 - |u^{(2)}|_0) + N (|v^{(2)}|_1 -  |u^{(2)}|_1 ) \\
					&   \ge  (d+1) \, C + (1-C) \, N = C \, (d+1-N) + N \ge d+1. 
			\end{split}
		\end{equation}
		This is impossible because $s \leq d + N$ and so
		\[
		|\sighat_1(v^{(2)})| - |\sighat_1(u^{(2)})| = |0^dv| -  |0^{N+d-s}u1|  =  s - N - 1 \le d-1.
		\]		
		This proves the claim.

		Note that the block $0^{d+N}$ appears in $\sighat_1(u^{(2)})$ (resp. $\sighat_1(v^{(2)})$), whenever a $1$ appears in $u^{(2)}$ (resp. $v^{(2)}$) because each $1$ is followed by $0$ in $u^{(2)}$ (resp. $v^{(2)}$). From the definition of $a$ we obtain 
		\begin{equation}\label{eq5}
			|u^{(2)}|_1 - |v^{(2)}|_1  =  |0^{d+N-s}u1|_{0^{d+N}} - |0^dv|_{0^{d+N}}  = (|u|_{10^{d+N}}+1) - |v|_{10^{d+N}}=a+1
		\end{equation}
		(the ``$+1$'' comes from the fact that the prefix $0^{d+N}$ of $0^{d+N-s}u1$ is not counted in $|u|_{10^{d+N}}$).
		
		Combining \eqref{eq:u1v1pos}, \eqref{eq4}, and \eqref{eq5} yields
		\begin{equation}\label{eq:qq}
			0\le |u^{(2)}| - |v^{(2)}| = a+1-C.
		\end{equation}
		Let $\widetilde{u}^{(2)}$ be the prefix of $u^{(2)}$ such that $|\widetilde{u}^{(2)}|=|v^{(2)}|$, {\it i.e.}, remove the last $a+1-C$ letters of $u^{(2)}$, of which at most $\left\lceil\frac{a+1-C}{d_2+1}\right\rceil$ are $1$'s, to obtain $\widetilde{u}^{(2)}$.  Suppose that $a \ge C+1$. Then from \eqref{eq5} we get 
		\[
		|\widetilde{u}^{(2)}|_1 - |v^{(2)}|_1 \geq a + 1 - \left\lceil \frac{a+1-C}{d_2+1} \right\rceil
		\geq C+1.
		\]
		This means that $(\widetilde{u}^{(2)},v^{(2)})$ is a not $C$-balanced pair of factors of $\ompr^{(2)}=\ompr(T_N(x),N)$ with $T_N(x) \in W_{K,N}$. But $(\widetilde{u}^{(2)},v^{(2)})$ has shorter length than $(u,v)$, contradicting the minimality of $(u,v)$. Together with \eqref{eq:qq} this yields $a\in \{C-1,C\}$. But if  $a=C$, we need to remove only the last letter  from $u^{(2)}$ (which is $0$) to obtain $\widetilde{u}^{(2)}$. Thus in this case we get from \eqref{eq5} that
		\[
		|\widetilde{u}^{(2)}|_1 - |v^{(2)}|_1 \geq a + 1 = C + 1
		\] 
		which is a contradiction again. Thus $a=C-1$ and we see from \eqref{eq3} that in this case we have $C-1 = \frac{(d+1)\,C + 1 - s}{N}$. Because $s\le d+N$ this implies that 
		\[
		\begin{split}
			C=\frac{s-N-1}{d+1-N}  \le \frac{d-1}{d+1-N} \le \frac{K-1}{K+1-N}.
		\end{split}
		\]
		Thus a minimal pair $(u,v)$ which is not $C$-balanced can only exist for $C \le \big\lfloor \frac{K-1}{K+1-N} \big\rfloor$. This contradicts our choice of $C$ and the theorem is proved.
		
		Statement (2) is straightforward from (1) and Lemma \ref{lem:NCbal}.
	\end{proof}	
	
	The previous theorem has an immediate corollary that gives balance constants that depend only on $N$.	
	
	\begin{corollary}\label{cor:N2balgen}\mbox{}
		\begin{enumerate}
			\item For all $x\in [0,1]\setminus \Q $ we have that	$\ompr(x,N)$ is $N$-balanced. Furthermore, for every $N\geq 2$ there are uncountable many $x\in [0,1]\setminus \Q $  such that $\ompr(x,N)$ is $2$-balanced.
			
			\item For all $x\in[0,1]\setminus \mathbb{Q}$ we have that	$\omega(x,N)$ is $N^2$-balanced. Furthermore, for every $N\geq 2$ there are uncountable many  $x\in[0,1]\setminus \mathbb{Q}$ such that $\omega(x,N)$ is $2N$-balanced.
		\end{enumerate}
	\end{corollary}
	
	\begin{proof}
		In (1), the first statement is immediate, since the greedy NCF expansion satisfies, for every $x\in [0,1]\setminus \Q $, that the digits are larger or equal to $N$, hence $x\in W_{N,N}$. The second statement follows from the fact that, for $x\in W_{2N,N}$, the dual $S$-adic sequence is $2$-balanced. Furthermore, it is not hard to check that $ W_{2N,N}$ is uncountable.
		
		Assertion (2) immediately follows from (1) and Lemma \ref{lem:NCbal}.		
	\end{proof}
	
	\subsection{Imbalance results}	
	We have given upper bounds for $C$ such that the $S$-adic sequences are $C$-balanced. Next we will give lower bounds. Note that in view of Corollary \ref{cor:N2balgen} this bound is optimal for $N=2$.
	\begin{proposition}
		Let $N\geq 2$ and $x\in[0,1]\backslash \Q$. Then the following assertions hold.
		\begin{enumerate}
			\item $\ompr(x,N)$ is not $1$-balanced.
			\item $\omega(x,N)$ is not $(2N-1)$-balanced. 
		\end{enumerate}
	\end{proposition}
	
	\begin{proof}
		Let $x=[0;d_1,d_2,\ldots]_N$.

        (1) Recall that $\ompr(x,N) = \sighat_1(\ompr^{(2)})$, where $\ompr^{(2)} = \ompr(T_N(x),N).$ It is easy to check that $10$ is a factor of $\ompr^{(2)}$, hence $\sighat_1(10) = 0^N0^{d_1}1$ is a factor of $\ompr(x,N)$. Since $N \geq 2,$ this yields the factor $u = 0^{d_1 + 2}$ of $\ompr(x,N)$.
        
        Analogously, $00$ is a factor of $\ompr^{(2)}$, hence $\sighat_1(00) = 0^{d_1}10^{d_1}1$ is a factor of $\ompr(x,N)$. This yields the factor $v = 10^{d_1}1$ of $\ompr(x,N)$. Because $u$ and $v$ are of the same length and  $|v|_1 - |u|_1 = 2$, assertion (1) is proved.

        (2) Recall that $\omega(x,N) = \sigma_1(\omega^{(2)}) = \sigma_1 \circ \sigma_2(\omega^{(3)})$ where $\omega^{(2)} = \omega(T_N(x),N)$ and $\omega^{(3)} = \omega(T_N^2(x),N).$ It is not hard to check that $001$ is a factor of $\omega^{(3)}$. We have 
        \begin{equation*}
            \begin{split}
                \sigma_1 \circ \sigma_2(001) &= \sigma_1 (0^{d_2} 1^N 0^{d_2} 1^N 0)\\
                    & = (0^{d_1}1^N)^{d_2} 0^N (0^{d_1}1^N)^{d_2} 0^N 0^{d_1} 1^N,
            \end{split}
        \end{equation*}
    which gives us the factor $u=0^N (0^{d_1}1^N)^{d_2} 0^{N+d_1}$ of $\omega(x,N)$. 
    
    On the other hand, since $N \geq 2$ it is not hard to check that $110$ is a factor of $\omega^{(3)}$. We have 
        \begin{equation*}
            \begin{split}
                \sigma_1 \circ \sigma_2 (110) &= \sigma_1(0^{d_2+2} 1^N)\\
                    & = 0^{d_1} 1^N (0^{d_1} 1^N)^{d_2} 0^{d_1} 1^N 0^N,
            \end{split}
        \end{equation*} 
    which gives us the factor $v = 1^N  (0^{d_1} 1^N)^{d_2} 0^{d_1} 1^N $ of $\omega(x,N)$. Since $|u| = |v|$ and $|v|_1 - |u|_1 = 2N$, assertion (2) follows.
    \end{proof}		
	\begin{remark}
		Suppose that, given $x \in [0,1],$ we consider infinite NCF expansions that are different from the greedy expansion, that is, we allow $d_n < N$. Note that such sequences also exist for rational values of $x$. For each expansion $(d_n)$ of $x$ we can then define the $S$-adic sequence $\widetilde{\omega}(x,(d_n),N)$ of the substitutions $(\sigma_n )$ associated with $d_n$ for all $n\geq 1$. Then we have the following.
		\begin{enumerate}			
			\item The property of $N^2$-balancedness does not hold in general for $\widetilde{\omega}(x,(d_n),N)$. For instance, let $N=2$ and $x =[0;1,1,1,\ldots]_2$, {\it i.e.}, the sequence of digits is given by $(d_n)=(1)$.  Then 
			$u= (0^31^2)^3(01^2)^20^31^20^3$ and 
			$v=(1^20)^30^2(1^20)^30^2(1^20)^21$
			are both factors of $\widetilde{\omega}(x,(d_n),N)$ of the same length and $|v|_1-|u|_1=5=N^2+1$.
			
			\item Suppose $(d_n)=(d)$ for some $d<N$. Then the associated substitution $\sigma$ is not Pisot, and hence the sequence $\widetilde{\omega}(x,N)$ is imbalanced, that is, it is not $C$-balanced for any $C>0.$ This is a consequence of \cite[Theorem 13]{A}.
			
			\item  $S$-adic sequences corresponding to eventually greedy NCF expansions are finitely balanced. Suppose $d_n\geq N$ for all $n \geq K$ for some $K\geq1$. Let $\sigma=\sigma_1\circ \dots \circ \sigma_{K-1}$ and let ${\tilde x} = [0;d_K,d_{K+1},\ldots]_N.$ Then $\widetilde{\omega}(x,(d_n),N)=\sigma(\omega({\tilde x},N))$ and $\omega({\tilde x},N)$ is $N^2$-balanced because it is an NCF sequence. It is not hard to check that, for any given substitution $\sigma$, if a sequence $\nu$ is $C$-balanced then $\sigma(\nu)$ is $C'$-balanced for some $C'$. Hence the statement follows.
		\end{enumerate}
	\end{remark}
	
	\section{Factor complexity of $N$-continued fraction sequences and unique ergodicity}\label{sec:complexity}
	
	\subsection{Definitions and preliminaries}
	Another property worth studying when working with sequences is their {\em factor complexity function}. This function counts how many different factors of a given length $n\in\N$ appear in a sequence.
\begin{definition}[Factor complexity function]
		Given a sequence $\nu\in \cA^\N$ over the finite alphabet $\cA$ and $n \in \N$, set
		\[\cL_{n}(\nu):=\{u\in\cA^*\,:\,|u|=n,\,u\mbox{ is a factor of }\nu \}.
		\]
		Define the \textit{factor complexity function} $p_{\nu} : \N \rightarrow \N$ as $p_{\nu}(n) = |\cL_{n}(\nu)|$.
\end{definition}
	
	The factor complexity function has the trivial upper bound $p_{\nu}(n) \leq |\cA|^n$, and $\nu$ is periodic if and only if there exists $n \in \N$ such that $p_{\nu}(n) \leq n$ (for the nontrivial direction see \cite[Corollary~4.3.2]{CN}). As mentioned in the introduction, a sequence $\nu$ is said to be Sturmian if it is not eventually periodic and its factor complexity function is given by $p_\nu(n) = n + 1$. This means that Sturmian sequences are the aperiodic sequences with the smallest possible complexity. We will show that NCF sequences and their duals have low complexity as well and we will explicitly compute a formula for their factor complexity function.
	We follow \cite{CN} and \cite{FMP}.
	
	A useful way to study the complexity function is to figure out how to obtain $\cL_{n+1}(\nu)$ from $\cL_{n}(\nu)$, and for that we make use of {\it left special factors}. Moreover, in what follows we will study three particular types of left special factors: infinite and maximal left special factors, and total bispecial factors. We introduce the corresponding definitions.
	
	\begin{definition}	
		Consider a sequence $\nu \in \cA^\N$ over a finite alphabet $\cA$. 
		\begin{enumerate}
			\item Given a factor $u$ of $\nu$, we say that a letter $a \in \cA$ is a \textit{left extension} of $u$ if $au$ is a factor of $\nu$. 
			
			\item We say that a factor $u$ of $\nu$ is  a \textit{left special factor} of $\nu$ if there are two distinct letters $a,b\in\cA$ that are left extensions of $u$. We set $$LS_n(\nu) :=\{u\in\cA^*\,:\,|u|=n,\,u\mbox{ is a left special factor of }\nu  \}.$$
		
			\item 	An infinite word $u \in \cA^\N$ is called an \textit{infinite left special factor} of $\nu$ if every prefix of $u$ is a left special factor of $\nu$. 
		
			\item A left special factor $u$ of $\nu$ is called a \textit{maximal left special factor} if $ua$ is not a left special factor for any $a \in \cA$.
		
			\item A left special factor $u$ of $\nu$ is called a \textit{total bispecial factor} if there exist $a,b\in\cA $ with $a \neq b$ such that $ua$ and $ub$ are both left special factors of $\nu$.
		\end{enumerate}
	\end{definition}
	In analogy with left extension, left special factor, infinite left special factor and maximal left special factor we can define right extension, right special factor, etc.
	
	\begin{definition}
	    A sequence $\nu \in \cA^\N$ over a finite alphabet $\cA$ is said to be \textit{recurrent} if every factor of $\nu$ occurs infinitely often.
	\end{definition}
	
	It is easy to see that our NCF sequences are recurrent by looking at \eqref{recurrencesigma} and \eqref{recurrencesighat}. 
	
	If we describe the occurrences of special factors in $\cL_{n}(\nu)$, we can determine $\cL_{n+1}(\nu)$. The following result is a direct consequence of \cite[Proposition 2.1]{FMP}.
	
	\begin{lemma}
		Let $\nu$ be a recurrent sequence over a two letter alphabet and let $p_\nu$ be its factor complexity function. Then, for every $n\in\N$, 		
		\begin{equation}\label{eq.mn}
			p_{\nu}(n+1) - p_{\nu}(n) = |LS_n(\nu)|.
		\end{equation} 
	\end{lemma}
	
	For example, a Sturmian sequence $\nu$ has exactly one left special factor of length $n$ for every $n$, which implies $p_{\nu}(n+1) - p_{\nu}(n) = 1$, and because $p_\nu(1)=2$ this implies that  $p_{\nu}(n) = n + 1$.
	
\subsection{Characterization of left special factors}
	
	In order to obtain the factor complexity function for dual NCF sequences, we will give a characterization of their left special factors. Fix $x \in [0,1] \setminus \Q$ and $N \geq 1$ and consider the dual NCF sequence $\ompr=\ompr(x,N).$ Recall that this is an $S$-adic sequence for the substitutions $\widehat{\bsigma}_x = (\sighat_n)_{n\geq 1},$ which means there exist infinite words $\ompr^{(1)}, \ompr^{(2)},\dots$ such that $\ompr = \ompr^{(1)}$ and $\ompr^{(n)} = \sighat_n(\ompr^{(n+1)})$ for every $n \geq 1$ (see Definition \ref{def:sadic}). This ``desubstitution" process is crucial for this section because we will show that many properties of special factors are invariant under the substitutions $\sighat_n$.
	
	\begin{lemma}\label{lsfdesubst} The following assertions hold for every $n \geq 1$.
		\begin{enumerate}
			\item Let $v$ be a left special factor of $\ompr^{(n+1)}$. Then $\sighat_n(v)$ is a left special factor of $\ompr^{(n)}$.
			
			\item Let $v$ be a left special factor of $\ompr^{(n)}$ ending in the letter $1$. Then there exists a left special factor $u$ of $\ompr^{(n+1)}$ such that $\sighat_n(u)=v$.
		\end{enumerate}
	\end{lemma}
	
	\begin{proof}
		(1) We have $ 0v, 1v \subset \ompr^{(n+1)} $. Then $ \sighat_n(0v) = 0^{d_n} 1 \sighat_n(v) $ and  $ \sighat_n(1v) = 0^{N} \sighat_n(v) $, which implies that $ \sighat_n(v) $ is a left special factor of $\ompr^{(n)}$.
		
		(2) Consider a word $u \subset \ompr^{(n+1)}$ of minimal length such that $v \subset \sighat_n(u)$. Since $v$ ends with $1$, by minimality of $u$ it holds that $v$ is a suffix of $\sighat_{n}(u)$. Since $v$ is a left special factor, $1v$ is a factor of $\ompr^{(n)}$, so it follows from the shape of the substitution $\sighat_n$ that $1v \subset \sighat_n(0u)$. The minimality of $u$ implies $\sighat_n(u)=v$. Since $0v$ is also a factor of $\ompr^{(n)}$, it turns out that $0v \subset \sighat_n(1u)$, and therefore $u$ is a left special factor of $\ompr^{(n+1)}$.
	\end{proof}
		
	 The following characterization of infinite left special factors is based on \cite[Section 3]{FMP}.
	
	\begin{lemma}\label{infinitelsf}
		The only infinite left special factor of $\ompr$ is $\ompr$ itself.
	\end{lemma}
	
	\begin{proof}
		Recall that $(\sipr_k)_{k\in\N}$ is a nested sequence of prefixes of $\ompr$ satisfying \eqref{recurrencesighat}. First, note that the last letter of $\sipr_k$ is congruent to $k+1$ modulo $2$. Also, for $k \geq 2$, both $ \sipr_k \sipr_k \subset \ompr$ and $ \sipr_{k-1} \sipr_k \subset \ompr$, giving us that $\sipr_k$ is a left special factor for every $k\geq1$. This implies that every prefix of $\ompr$ is a left special factor of $\ompr$ and hence $\ompr$ is an infinite left special factor of itself.
		
		It remains to show uniqueness. Let $\nu$ be an infinite left special factor of $\ompr$. Recall that $\ompr = \sighat_1(\ompr^{(2)})$. We show first that there exists an infinite left special factor of $\ompr^{(2)}$, namely $\nu^{(2)}$, such that $\nu = \sighat_1(\nu^{(2)})$. Consider a prefix $v$ of $\nu$ ending in the letter $1$. By hypothesis, this prefix is a left special factor. By part $(2)$ of Lemma~\ref{lsfdesubst}, there exists a left special factor $v^{(2)}$ of $\ompr^{(2)}$ such that $\sighat_1(v^{(2)})=v$.	The prefix $v$ can be chosen to be arbitrarily large, which means $v^{(2)}$ can be made to be arbitrarily large. This implies the existence of an infinite left special factor $\nu^{(2)}$ of $\ompr^{(2)}$ such that $\nu = \sighat_1(\nu^{(2)})$. Iterating this reasoning, we can obtain a sequence $(\nu^{(n)})_{n\geq 1}$, $\nu^{(1)}=\nu$, such that for each $n$, $\nu^{(n)}$ is an infinite left special factor of $\ompr^{(n)}$ and $\sighat_{n}(\nu^{(n+1)}) = \nu^{(n)}$. We use this to show that $\nu = \ompr$. 
		
		Suppose $\nu \neq \ompr$, then $\ompr^{(n)}\neq \nu^{(n)}$ for each $n\geq 1$. This enables the definition of 
		\[ d(\ompr^{(n)}, \nu^{(n)}) := \min \{ k\,:\, \ompr^{(n)}_k\neq \nu^{(n)}_k\}. \]
		Note that the substitutions $\sighat_{n}$ strictly increase the length of a word unless the word is $1$ and $N=1$. Since the image under $\sighat_n$ of both letters starts with $0$, it follows that $1$ is not a prefix of $\ompr^{(n)}$ nor of $\nu^{(n)}$ for any $n \geq 1$. Therefore, it holds for all $n \geq 1$ that $d(\ompr^{(n)}, \nu^{(n)}) > 1$ and $d(\ompr^{(n+1)},\nu^{(n+1)}) < d(\ompr^{(n)}, \nu^{(n)}),$ so the sequence of distances $(d(\ompr^{(n)}, \nu^{(n)}))_{n\in\N}$ is strictly decreasing yet strictly positive, a contradiction. This shows that $\nu = \ompr$.	
\end{proof}
	
	Next, we state some lemmas regarding maximal left special factors and total bispecial factors.

	\begin{lemma}\label{lsfmaxdesubst}
		Let $v$ be a maximal left special factor of $\ompr^{(n)}$ containing the letter $1$. Then there exists a maximal left special factor $u$ of $\ompr^{(n+1)}$ such that $ v = \sighat_n(u) 0^{d_n}$.
	\end{lemma}
	
	\begin{proof}
		First, note that if a left special factor has a unique right extension it cannot be maximal. Given a maximal left special factor $v$, since it has more than one right extension and contains a $1$ by assumption, it must be of the form $ v = v_0 v_1 \cdots v_s 1 0^{d_n} $ for some $s \in \N$ (see part (2b) of Lemma~\ref{lem:maxFactor}). By part $(2)$ of Lemma \ref{lsfdesubst}, there exists a left special factor $u$ of $\ompr^{(n+1)}$ such that $ v = \sighat_n(u) 0^{d_n} $. It remains to show that $u$ is maximal. Suppose it is not, then there exists a letter $ a \in \cA $ such that $ ua $ is also a left special factor. By part $(1)$ of Lemma \ref{lsfdesubst}, $\sighat_n( ua )$ is a left special factor of $\ompr^{(n)}$. If $a = 0$ then $v$ is a proper prefix of the left special factor $\sighat_n( u0 )$, contradicting the maximality of $v$. If $a = 1$ then $u10$ is a left special factor because $1$ is always followed by $0$, and in this case $v$ is also a proper prefix of the left special factor $\sighat_n( u10 )$, contradicting again the maximality of $v$.
	\end{proof}
	
	\begin{lemma}\label{tbfdesubst}
		Let $v$ be a total bispecial factor of $\ompr^{(n)}$ containing the letter $1$. Then there exists a total bispecial factor $u$ of $\ompr^{(n+1)}$ such that $ v = \sighat_n(u) 0^{d_n}$.
	\end{lemma}
	
	\begin{proof}
		Since $v$ has more than one right extension and contains the letter $1$ it must be of the form $ v = v_0 v_1 \cdots v_s 1 0^{d_n} $ for some $s \in \N$ (see part (2b) of Lemma~\ref{lem:maxFactor}). By part $(2)$ of Lemma \ref{lsfdesubst}, there exists a left special factor $u$ of $\ompr^{(n+1)}$ such that $ v = \sighat_n(u) 0^{d_n} $. It remains to show that $u$ is a total bispecial factor. By hypothesis, $0v0,0v1,1v0,1v1 \subset \ompr^{(n)}$.  We have
		\[
		\begin{split}
			0 \sighat_n(u) 0^{d_n + 1} & \subset \sighat_n(1u10), \\
			0 \sighat_n(u) 0^{d_n} 1 & \subset \sighat_n(1u0), \\
			1 \sighat_n(u) 0^{d_n + 1} & \subset \sighat_n(0u10), \\
			1 \sighat_n(u) 0^{d_n} 1 & \subset \sighat_n(0u0). \\
		\end{split}
		\]
		It follows from the shape of $\sighat_n$ that $0 \sighat_n(u)$ can only occur as a factor of $\sighat_n(1u)$ and $1 \sighat_n(u)$ can only occur as a factor of $\sighat_n(0u)$; moreover, $\sighat_n(u) 0^{d_n + 1}$ can only occur as a factor of $\sighat_n(u10)$ and $\sighat_n(u) 0^{d_n} 1$ can only occur as a factor of $\sighat_n(u0)$.
		This implies that $u$ is a total bispecial factor of $\ompr^{(n+1)}$.
	\end{proof}
	
	\subsection{Results on factor complexity}
	Fix $N\geq 2$ and $x\in[0,1]\setminus\Q$. Consider the NCF sequence $\omega(x,N)$ and its dual $\ompr(x,N)$. Consider the corresponding nested sequences of prefixes $(\Sigma_n)_{n\geq 1}$ and $(\sipr_n)_{n\geq 1}$ and the respective recurrence relations \eqref{recurrencesigma} and \eqref{recurrencesighat}.
	
	Define, for each $k\geq 1$, the words 
	\[ \Shat_k := \sipr_k^{d_k} \sipr_{k-1}^{d_{k-1}} \cdots \sipr_1^{d_1}, \qquad \That_k := \sipr_k^{N-1} \Shat_k \]
	and
	\[ S_k := { \Sigma}_k^{d_k} \Sigma_{k-1}^{d_{k-1}} \cdots {\Sigma}_1^{d_1},\qquad T_0:=1^{N-1},\qquad T_k := {\Sigma}_k^{ N-1} S_k.\]
	
	Define the numbers $\smallthat_0 < \smallshat_1 < \smallthat_1 < \smallshat_2 < \smallthat_2 < \cdots$ as 
	\[ 
	\smallshat_k := |\Shat_k|,\quad \smallthat_0=0,\quad \smallthat_k := |\That_k|\qquad (k\geq1).  
	\]
	
	Define the numbers $t_0 < s_1 < t_1 < s_2 < t_2 < \cdots $ as 
	\[ 
	s_k:=|S_k|\quad (k\geq1),\qquad  t_k:=|T_k|\quad (k\geq0). 
	\]
	Recall the definition of ${p_n}$ and ${q_n}$ with $n\geq -1$ given in \eqref{convergents}. The main result of this section is the following.
	
	\begin{theorem}\label{thm:complexity}
		The factor complexity functions of $\omega=\omega(x,N)$ and $\ompr=\ompr(x,N)$ satisfy 
\[
p_\omega(n) \leq 2n \qquad\text{and}\qquad p_\ompr(n) \leq 2n \qquad\qquad(n \geq 1).
\]
In particular, they are given by
		\begin{equation}\label{complexity}
			p_{\omega}(n)=
			\begin{cases}
				1, & n = 0, \\
				2n,& 1\leq n\leq N-1,\\
				n + 1 + \sum_{j=-1}^{k-1}(p_{j} + q_{j})(N - 1) ,& t_k< n \leq s_{k+1},\\
				2n + 1 + \sum_{j=-1}^{k-1} (p_{j} + q_{j}  ) (N - 1) - s_k,& s_k < n \leq t_k
				
			\end{cases}
		\end{equation}
		and
		\begin{equation}\label{complexityprime}
			p_{\ompr}(n)=
			\begin{cases}
				1, & n = 0, \\
				n + 1 + \sum_{j=0}^{k-2}(\frac{p_j}{N}+q_j)(N - 1) ,& \smallthat_{k-1}< n \leq \smallshat_k,\\
				2n + 1 + \sum_{j=0}^{k-2}(\frac{p_j}{N}+q_j)(N - 1) - \smallshat_k,& \smallshat_k < n \leq \smallthat_k.\\
			\end{cases}
		\end{equation}
	\end{theorem}
	
	Before proceeding to the proof, we state and prove some lemmas.
	
	\begin{lemma}\label{usuffix}
		If $v$ is a maximal left special factor of $\ompr=\ompr(x,N)$, then it is of the form $\That_k$ for some $k\geq 1$. 
	\end{lemma}
	
	\begin{proof}
		Let $ v = v^{(1)} $ be a maximal left special factor of $\ompr$. Then either it contains a $1$ or not. Suppose it does not contain a $1$. It is not hard to check that the only maximal left special factor of $\ompr^{(n)}$ that does not contain the letter $1$ is $0^{d_n + N - 1}=\That_1$. Now suppose it does contain a $1$. By Lemma~\ref{lsfmaxdesubst}, we can write $ v^{(1)} = \sighat_1 (v^{(2)}) 0^{d_1}$ for a maximal left special factor $v^{(2)} \subset \ompr^{(2)}$. If $v^{(2)}$ also contains a $1$, we can desubstitute again and obtain a maximal left special factor $v^{(3)} \subset \ompr^{(3)}$ such that $ v^{(2)} = \sighat_2 (v^{(3)}) 0^{d_2}$. Iterating this process, because at each desubstitution the length of the words decreases, we eventually reach a maximal left special factor $v^{(n)}$ of $\ompr^{(n)}$ that does not contain the letter $1$ which is $0^{d_n + N - 1}$. We get
\begin{equation}\label{thatdesubst}
\begin{split}
v 	& = \sighat_1 ( \sighat_2 ( \cdots \sighat_{n-1} ( 0^{d_n + N - 1}) 0^{d_{n-1}} \cdots ) 0^{d_2} ) 0^{d_1}\\
& = ( \sighat_1 \circ \cdots \circ \sighat_{n-1} (0) )^{ d_n + N - 1 } ( \sighat_1 \circ \cdots \circ \sighat_{n-2} (0) )^{d_{n-1} } \cdots 0^{d_1} \\
	& = \sipr_n^{d_n + N - 1} \sipr_{n-1}^{d_{n-1}} \cdots \sipr_1^{d_1} = \That_n
\end{split}
\end{equation}
and the result follows.
	\end{proof}
	
	We now have a complete understanding of maximal left special factors of $\ompr$. It is clear that any left special factor is either a prefix of $\ompr$ (the only infinite left special factor) or a prefix of $\That_k$ for some $k \geq 1$. The following result follows from the definition of $\Shat_k$ and of $\sipr_k$ and uses our previous results on total bispecial factors.
	
	\begin{lemma}\label{commonpref}
		For each $k \geq 1 $, the maximal common prefix of $\That_k$ and $\ompr$ is $\Shat_k$.
	\end{lemma}
	
	\begin{proof}
		Fix $k \geq 1$ and suppose $v$ is the maximal common prefix of $\That_k$ and $\ompr$. Then $v$ is a strict prefix of $\That_k$ because $\That_k$ is not a prefix of $\ompr$. By maximality, we have that there are  $a,b\in \{0,1\}$ with  $a \neq b$  such that $va$ is a prefix of $\That_k$ and $vb$ is a prefix of $\ompr$. Because all prefixes of $\That_k$ and all prefixes of $\ompr$ are left special factors, $v$ is a total bispecial factor of $\ompr$. If $k = 1$ then $\That_1 = 0^{d_1 + N - 1}$ and $v = \Shat_1 = 0^{d_1}$. If $k \geq 2$ then $v$ contains the letter $1$. By Lemma~\ref{tbfdesubst},  we can write $ v = v^{(1)} = \sighat_1 (v^{(2)}) 0^{d_1}$ for a total bispecial factor $v^{(2)} \subset \ompr^{(2)}$. If $v^{(2)}$ also contains a $1$, we can desubstitute again and obtain a total bispecial factor $v^{(3)} \subset \ompr^{(3)}$ such that $ v^{(2)} = \sighat_2 (v^{(3)}) 0^{d_2}$. Iterating this process, because at each desubstitution the length of the words decreases, we eventually reach a total bispecial factor $v^{(n)}$ of $\ompr^{(n)}$ that does not contain the letter $1$. It is not hard to check that it must be of the form $v^{(n)} = 0^{d_n}$. Since $v$ is a prefix of $\That_k$, using the desubstitution of $\That_k$ given in \eqref{thatdesubst}, we have $n \leq k.$ By maximality of $v$, we obtain $n = k$. We get
		\begin{equation*}
			\begin{split}
				v 	& = \sighat_1 ( \sighat_2 ( \cdots \sighat_{k-1} ( 0^{d_k}) 0^{d_{k-1}} \cdots ) 0^{d_2} ) 0^{d_1}\\
				& = ( \sighat_1 \circ \cdots \circ \sighat_{k-1} (0) )^{ d_k } ( \sighat_1 \circ \cdots \circ \sighat_{k-2} (0) )^{d_{k-1} } \cdots 0^{d_1} \\
				& = \sipr_k^{d_k} \sipr_{k-1}^{d_{k-1}} \cdots \sipr_1^{d_1} = \Shat_k.\qedhere
			\end{split}
		\end{equation*}
	\end{proof}
	
	The following lemma computes the difference between consecutive terms of $p_{\ompr}$ using left special factors.
	
	\begin{lemma}\label{prop.complexity}
	Let $\ompr=\ompr(x,N)$ with factor complexity function $p_\ompr$. Then,  for every $n \geq 1$,
		
		\begin{equation}\label{difcomplexity}
			p_{\ompr}(n+1) - p_{\ompr}(n)=
			\begin{cases}
				1,& \smallthat_{k-1}< n \leq \smallshat_k,\\
				2,& \smallshat_k < n \leq \smallthat_k.\\
			\end{cases}
		\end{equation}
	\end{lemma}
	
	\begin{proof}
		Let $n\geq 1$. By Lemma~\ref{infinitelsf}, the prefix of $\ompr$ of length $n$ is a left special factor. We have that 
		\[ |LS_n(\ompr)| = 1 + |LS^*_n(\ompr)|, \] 
		where $ LS^*_n(\ompr) $ is the set of left special factors that are not prefixes of $ \ompr $. Given a left special factor $v$ of length $n$, Lemma~\ref{usuffix} implies it is a prefix of $\That_k$ for some $k$. Take $k$ so that $\smallthat_{k-1} < n \leq \smallthat_k$. Suppose $\smallthat_{k-1} < n \leq \smallshat_k$, then $v$ is a prefix of $\Shat_k$. It follows from Lemma~\ref{commonpref} that $v$ is a prefix of $\ompr$, therefore $ LS^*_n(\ompr) = \varnothing $ and so $|LS_n(\ompr)|=1$.

		Suppose now that $\smallshat_k < n \leq \smallthat_k$. In this case, $\That_k$ and $\ompr$ do not coincide in the first $n$ letters, hence $|LS^*_n(\ompr)| = 1 $ and therefore $|LS_n(\ompr)|=2$.
		
		Then
		\[ 
		|LS_n(\ompr)|=
		\begin{cases}
			1,& \smallthat_{k-1}< n \leq \smallshat_k,\\
			2,& \smallshat_k < n \leq \smallthat_k.\\
		\end{cases}\quad
		\] Then \eqref{difcomplexity} follows from \eqref{eq.mn}.
	\end{proof}
	
	We are in position to prove the main theorem of this section.
	
	\begin{proof}[Proof of Theorem \ref{thm:complexity}]
		We first compute the factor complexity function of $\ompr$. Using that~$p_\ompr(0)=1$ (then only word of length $0$ is $\varepsilon$) and $p_\ompr(1)=2$ (the alphabet has two letters), \eqref{complexityprime} follows from \eqref{difcomplexity} by direct computation. Indeed, let $n \geq 1$ and suppose first that $\smallthat_{k-1} < n \leq \smallshat_k$ for some $k$. We obtain from Lemma~\ref{prop.complexity} that 
		\[
		\begin{split}
		p_\ompr(n) &= 1 + (\smallshat_1 - \smallthat_0) + 2(\smallthat_1 - \smallshat_1) + (\smallshat_2 - \smallthat_1)+\cdots + 2(\smallthat_{k-1} - \smallshat_{k-1})+ (n-\smallthat_{k-1}) \\ &= n + 1 + (\smallthat_1 - \smallshat_1) + \cdots + (\smallthat_{k-1} - \smallshat_{k-1}).\\
		\end{split}\]
		Note that  $\smallthat_j= |\sipr_j|(N -1) + \smallshat_j$ for all $j\ge 1$. Hence,
		\[ p_\ompr(n)= n + 1 + \sum_{j=1}^{k-1}|\sipr_j|(N - 1) . \]
		Suppose now that $\smallshat_k < n \leq \smallthat_k$. Then
		\[
		\begin{split}
		p_\ompr(n) &= 1 + (\smallshat_1 - \smallthat_0) + 2(\smallthat_1 - \smallshat_1) + (\smallshat_2 - \smallthat_1)+\cdots +  (\smallshat_k - \smallthat_{k-1}) + 2( n - \smallshat_k) \\ &= 2n + 1 + (\smallthat_1 - \smallshat_1) + \cdots + (\smallthat_{k-1} - \smallshat_{k-1}) - \smallshat_k\\
		&=2n + 1 + \sum_{j=1}^{k-1}|\sipr_j|(N - 1) - \smallshat_k.
		\end{split}\]
		Recall from $\eqref{eq:letterfreq}$ that, given $j \geq 0$, $|\Sigma_j|_0 = q_{j-1}$ and $|\Sigma_j|_1 = p_{j-1}$, hence $|\Sigma_j|=p_{j-1}+q_{j-1}.$ Recall also that, 	using the substitution $\tau$ defined in \eqref{eq:sdef} that maps $\ompr$ to $\omega$, there is a factor $1^N$ in $\Sigma_j$ whenever there is a $1$ in $\sipr_j$. Thus $|\sipr_j|_0 = |\Sigma_j|_0$ and $|\sipr_j|_1 = \frac{|\Sigma_j|_1}{N}$, and therefore $|\sipr_j|=\frac{p_{j-1}}{N}+q_{j-1}.$ This yields \eqref{complexityprime}.
				
		We can obtain the complexity of $\omega$ just like we did for $\ompr$. We have $p_\omega(0)=1$ and $p_\omega(1)=2$. Note that $T_0 = 1^{N-1}$ is a maximal left special factor of $\omega$. If $1 \leq n \leq N-1$, it holds that $LS_n(\omega) = \{ 0^n, 1^n \}$, that is, there are two left special factors; this implies $p_\omega(n) = 2n$. Note that, if a word $u \subset \ompr$ starts with $0$, then $au$ is a left extension of $u$ in $\ompr$ if and only if $a \tau(u)$ is a left extension of $\tau(u)$ in $\omega$. Hence, it is easy to check that $u$ is a left special factor of $\ompr$ if and only if $\tau(u)$ is a left special factor of $\omega$. Moreover, $u$ is a maximal left special factor of $\ompr$ if and only if $\tau(u)$ is a maximal left special factor of $\omega$. Also, $\omega$ is the unique infinite left special factor of itself, by the same arguments as in the proof of Lemma~\ref{infinitelsf}. Moreover, Lemmas \ref{usuffix} and \ref{commonpref} are analogous for $\omega$: the maximal left special factors of $\omega$ are the words $T_k$ for $k \geq 1$ (to which we add $T_0 = 1^{N-1}$), and the maximal common prefix of $\omega$ and $T_k$ is $S_k$. Proceeding like in the proof of Lemma~\ref{prop.complexity}, this yields, for $n \geq N$, 
		\begin{equation}\label{difcomplexityprime}
			|LS_n(\omega)| = 
			\begin{cases}
				1, & \quad t_k < n \leq s_{k+1},\\
				2, & \quad s_k < n \leq t_k.				
			\end{cases}
		\end{equation}
		Finally, \eqref{complexity} follows from \eqref{difcomplexityprime}. 
	\end{proof}

\subsection{Uniform word frequencies and unique ergodicity}
	Fix $N\geq 2$ and $x\in[0,1]\setminus\Q$. Consider the NCF sequence $\omega(x,N)$ and its dual $\ompr(x,N)$. Our results on word combinatorics can contribute to the study of associated symbolic dynamical systems. This will now be exploited.
	
	\begin{definition}[Unique ergodicity]
		We say that a topological dynamical system is uniquely ergodic if it has only one invariant probability measure.
	\end{definition}

	\begin{corollary}
		The topological dynamical systems $(X_{\omega(x,N)},\Sigma)$ and $(X_{\ompr(x,N)},\Sigma)$ are uniquely ergodic.
	\end{corollary}
	\begin{proof}
		A criterion of Boshernitzan states that unique ergodicity for a minimal subshift $(X_\nu,\Sigma)$ holds if the factor complexity function satisfies $\limsup \frac {p_\nu(n)}{n}<3.$ By \cite[Proposition 7.1.5]{FM}, the subshift $(X_\nu,\Sigma)$ is minimal if and only if $\nu$ is uniformly recurrent, that is, if every factor of $\nu$ occurs in an infinite number of places with bounded gaps.
		
		Any factor $u$ of $\omega(x,N)$ is a factor of $\Sigma_k$ for some $k$, and each $\Sigma_k$ appears infinitely often with bounded gaps, hence $\omega(x,N)$ is uniformly recurrent and hence the subshift $(X_\omega,\Sigma)$ is minimal. Since $p_\omega(n) \leq 2n$, the assertion follows. The reasoning for $\ompr(x,N)$ is analogous. 
	\end{proof}

	\begin{definition}[Uniform word and letter frequency]
		Consider a sequence $\nu=\nu_0\nu_1\cdots \in \cA^\N$ with $\nu_i \in \cA.$ Recall that for each $k,l\in\N$ and $u\in\cA^*$, we denote by $|\nu_k\cdots \nu_{k+l-1}|_u$ the number of occurrences of the word $u$ in the factor $\nu_k\cdots \nu_{k+l-1} \subset \nu$. We say that $\nu$ has uniform word frequency if for each $u\in\cA^*$ there exists $f_\nu(u)\in\R$ which does not depend on $k$ such that
		\[ \lim_{l\to\infty}\frac{|\nu_k\cdots \nu_{k+l-1}|_u}{l}=f_\nu(u).
		\]
	\end{definition}
	
	\begin{corollary}\label{uniformword}
		The continued fraction word $\omega(x,N)$ and its dual $\ompr(x,N)$ have uniform word frequency. Moreover, this holds for all the elements of $X_{\omega(x,N)}$ and $X_{\ompr(x,N)}$.
	\end{corollary}
	\begin{proof}
		This result follows from \cite[Theorem 7.2.10]{FM} since the respective dynamical systems are uniquely ergodic. 
	\end{proof}

	\section{Entropy, growth rate, and a Farey map for NCF sequences}\label{sec:dynamics}
	
	\subsection{Entropy and growth rate}
	
We now calculate how fast the words $\subw{n}= \sigma_1\circ\sigma_2\circ\cdots\circ \sigma_n(1)$ grow when $n$ tends to $\infty$. To this end we need the entropy $h(T_N)$ of the $N$-continued fraction map $T_N$. This entropy, which is calculated in~\cite{DKW}, is given by
	\[
	h(T_N)=\frac{\frac{\pi^2}{3}+2Li_2(N+1)+\log(N+1)\log(N)}{\log\left(\frac{N+1}{N}\right)}
	\]
	where $Li_2$ denotes the dilogarithm function defined by 
	\[
	Li_2(x)=\int_0^x \frac{\log(t)}{1-t}dt.
	\]
Our result reads as follows.

\begin{proposition}
For the growth rate of the lengths of the words $\subw{n}=\sigma_1\circ\sigma_2\circ\cdots\circ \sigma_n(1)$ we obtain the formula
\begin{equation}
		\lim_{n\to\infty}\frac{1}{n}\log(|\subw{n}|)=\frac{1}{2}\left(h(T_N)+\log(N)\right).
\end{equation}
\end{proposition}

\begin{proof}
Recall that $|\subw{n}|_1=p_{n-1}$ and $|\subw{n}|_0=q_{n-1}$. Then 
	\begin{eqnarray*}
		\lim_{n\to\infty}\frac{1}{n} \log(|\subw{n}|)&=&\lim_{n\to\infty}\frac{1}{n}\log(p_{n-1}+q_{n-1})\\
		&=&\lim_{n\to\infty}\frac{1}{n}\log\left(q_{n-1}\left(\frac{p_{n-1}}{q_{n-1}}+1\right)\right)\\
		&=&\lim_{n\to\infty}\frac{1}{n}\log(q_n).
	\end{eqnarray*}
	Thus it remains to show that
	\begin{equation}\label{eq:entqn}
		\lim_{n\to\infty}\frac{1}{n}\log(q_n)=  \frac{1}{2}\left(h(T_N)+\log(N)\right),
	\end{equation}
	where $h(T_N)$ is the entropy of $T_N$. 	Let $$\Delta(d_1,d_2,\ldots, d_n)=\{y\in [0,1] \,:\, y=[0;d_1,d_2,\ldots, d_n,\ldots]_N \}$$ be cylinders of order $n$ and $$\Delta_n(x)=\{y\in [0,1] \,:\, y=[0;d_1(x),d_2(x),\ldots, d_n(x),\ldots]_N \}.$$ Since the cylinders are a finite generator for $T_N$ and give a finite countable partition, we have for almost every $x\in [0,1]$ that 
	\begin{equation}\label{eq:SMBC}
		\lim_{n\to \infty} -\frac{1}{n}\log \left(\mu_N(\Delta_n(x))\right)=h(T_N)
	\end{equation}
	by using the Shannon-McMillan-Breiman-Chung Theorem and a Theorem of Kolmogorov and Sinai (see \cite{DK}, Chapter 6). Here $\mu_N$ is the absolutely continuous invariant measure. Since the measure $\mu_N$ and the Lebesgue measure $\lambda$ are equivalent we have
	\[
	\lim_{n\to \infty} \frac{1}{n}\log \left(\mu_N(\Delta_n(x))\right)=\lim_{n\to \infty} \frac{1}{n}\log \left(\lambda(\Delta_n(x))\right).
	\]
	Now, similar to the cylinders of the regular continued fraction (follow the same lines as \cite{DK}, page 28), we have
	\[
	x=\frac{p_n+p_{n-1}T_N^n(x)}{q_n+q_{n-1}T_N^n(x)}
	\]
	which gives us that the endpoints of $\Delta_n(x)$ are given by $\frac{p_n+p_{n-1}}{q_n+q_{n-1}}$ and $\frac{p_n}{q_n}$. From this and $q_np_{n-1}-p_nq_{n-1}=(-N)^n$ it follows that  
	\[
	\lambda(\Delta_n(x))=\left|\frac{p_n+p_{n-1}}{q_n+q_{n-1}}-\frac{p_n}{q_n}\right|=\frac{N^n}{q_n(q_n+q_{n-1})}.
	\]
	We find
	\begin{eqnarray*}
		\lim_{n\to \infty} \frac{1}{n}\log \left(\lambda(\Delta_n(x))\right)
		&=& \lim_{n\to \infty} \frac{1}{n}\log \left(\frac{N^n}{q_n(q_n+q_{n-1})}\right)\\
		&=& \log(N)- \lim_{n\to \infty} \frac{1}{n}\log \left(q_n(q_n+q_{n-1})\right)\\
		&=& \log(N)- \lim_{n\to \infty} \frac{1}{n}\log \left(q_n^2\left(1+\frac{q_{n-1}}{q_n}\right)\right)\\
		&=& \log(N)- \lim_{n\to \infty} \frac{2}{n}\log (q_n)\\
	\end{eqnarray*}
	which gives us (\ref{eq:entqn}).
\end{proof}

	\subsection{A Farey-like map for greedy NCF expansions}\label{sec:farey}
	Just as in the case of the Gauss map, there is also a slow version of the map $T_N$. Let $F_N:[0,1]\rightarrow[0,1]$ be defined as
	\[
	F_N(x)=
	\begin{cases}
		\frac{N}{x}-N & \text{for } x\in \left(\frac{N}{N+1},1\right],\\
		\frac{Nx}{N-x} & \text{for } x\in \left[0,\frac{N}{N+1}\right];
	\end{cases}
	\]
	see Figure \ref{fig:F2}.
	
	\begin{figure}[ht]
		\centering
		\subfigure{\begin{tikzpicture}[scale=5]
				\draw(0,0)node[below]{\small $0$}--(1,0)node[below]{\small $1$}--(1,1)--(0,1)node[left]{\small $1$}--(0,0);

				\draw[thick, purple!50!black, smooth, samples =20, domain=2/3:1] plot(\x,{2/\x-2});
				\draw[thick, purple!50!black, smooth, samples =20, domain=0:2/3] plot(\x,{2*\x/(2-\x)});

				\draw[dotted](2/3,0)node[below]{\small $\frac{2}{3}$}--(2/3,1);
				
		\end{tikzpicture}}
		\subfigure{\begin{tikzpicture}[scale=5]
				\draw[white] (-0.25,0)--(1,0);
				\draw(0,0)node[below]{\small $0$}--(1,0)node[below]{\small $1$}--(1,1)--(0,1)node[left]{\small $1$}--(0,0);

				\draw[thick, purple!50!black, smooth, samples =20, domain=5/6:1] plot(\x,{5/\x-5});
				\draw[thick, purple!50!black, smooth, samples =20, domain=0:5/6] plot(\x,{5*\x/(5-\x)});

				\draw[dotted](5/6,0)node[below]{\small $\frac{5}{6}$}--(5/6,1);
				
		\end{tikzpicture}}
		
		\caption{The map $F_N$ for $N=2$ on the left and $N=5$ on the right. \hspace{5.5em} }
		\label{fig:F2}
	\end{figure}
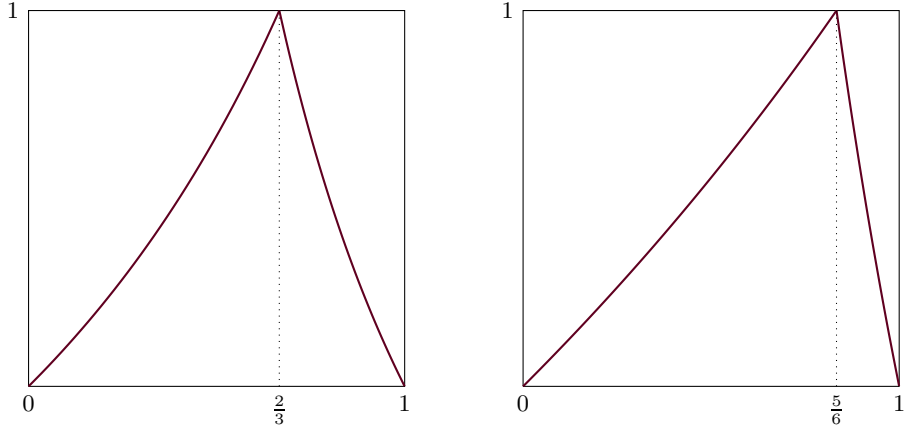

	Now let $x=[0;d_1,d_2,d_3,\ldots]_N$. It is clear that if $d_1=N$ then $F_N(x)=[0;d_2,d_3,\ldots]_N$. If $d_1>N$ we have
	\[
	F_N(x)=\frac{Nx}{N-x}=\frac{N}{\frac{N}{x}-1}=\frac{N}{d_1-1+\displaystyle\frac{N}{d_2+\ddots}}
	\]
	so that $F_N(x)=[0;d_1-1,d_2,d_3\ldots]_N$. This justifies the name Farey-like map.
	It is easy to check that the maps $F_N$ are AFN-maps which are considered in \cite{Z}. Then by following the same arguments as in \cite{KLMM}, every open interval contains a rational number and all rational numbers  are eventually mapped to the indifferent fixed point, we can conclude that for every $N$ there exists a unique absolutely continuous, infinite, $\sigma$-finite $F_N$-invariant measure $\mu_N$ that is ergodic and conservative for $F_N$. We now show the following result. 
\begin{proposition}
The infinite measure $\frac{dx}{x}$ is an invariant measure for the dynamical system $(F_N,[0,1],\mathcal{B})$. Here $\mathcal{B}$ is the Borel $\sigma$-algebra on $[0,1]$. 
\end{proposition}

\begin{proof}
It suffices to show invariance for intervals $(a,b)\subset[0,1]$. Now
	\begin{equation}
		\mu\left((a,b)\right)=\int_a^b \frac{1}{x} dx= \log\left(\frac{b}{a}\right).
	\end{equation}
	We have that $F_N^{-1}\left((a,b)\right)=\left(\frac{Na}{N+a},\frac{Nb}{N+b}\right)\cup \left(\frac{N}{N+b},\frac{N}{N+a}\right)$ which gives us
	\begin{eqnarray*}
		\mu\left(F_N^{-1}(a,b)\right)&=&\int_{\frac{Na}{N+a}}^{\frac{Nb}{N+b}} \frac{1}{x} dx+\int_{\frac{N}{N+b}}^{\frac{N}{N+a}} \frac{1}{x}dx\\
		&=& \log\left(\frac{\frac{Nb}{N+b}}{\frac{Na}{N+a}}\right)+\log\left(\frac{\frac{N}{N+a}}{\frac{N}{N+b}}\right)\\
		&=& \log\left(\frac{Nb(N+a)}{Na(N+b)}\frac{N(N+b)}{N(N+a)}\right)=\log\left(\frac{b}{a}\right)
	\end{eqnarray*}
	which finishes the proof. 
\end{proof}

To the dynamical system $(F_N,[0,1],\mathcal{B})$ we can also associate $S$-adic sequences in the same way as we did for the dynamical system defined by $T_N$. Note that if $x=[0;d_1,d_2,d_3,\ldots]_N$ then the sequence $1^{d_1-N},N,1^{d_2-N},N,1^{d_3-N},N,\ldots$ is the slow expansion of $x$ corresponding to $F_N$. 
If we want a strong analogy with the NCF sequences that we studied, we should take the the  M\"obius transformations corresponding to the inverse branches of $F_N$ and swap the numbers on the diagonal and anti-diagonal. This way we find 
	\begin{equation}\label{BandD}
		B=\left(
		\begin{matrix}
			N & 1 \\
			0 & N 
		\end{matrix}
		\right)
		\quad \text{and}  \quad 
		D=\left(
		\begin{matrix}
			N & 1 \\
			N & 0 
		\end{matrix}
		\right).
	\end{equation}
	These matrices should be the incidence matrices of the substitutions. With this in mind we associate the slow NCF expansions with the following substitutions
	\[ 
	\tau_{B} :
	\begin{cases}
		0\rightarrow 0^N,\\		
		1\rightarrow 0 1^N , \\
	\end{cases}
	\]
	and
	\[ 
	\tau_{D} :
	\begin{cases}
		0\rightarrow 0^N 1^N.\\		
		1\rightarrow 0,\\
	\end{cases}
	\]
	For every irrational number $x$ we can now find a directive sequence $\bsigma=(\sigma_n)_{n\geq 1}$, where $\sigma_n=\tau_{B}$ if the $i^{\text{th}}$ digit in the slow expansion is $1$ and $\sigma_n=\tau_{D}$ if it is $N$. It would be interesting to study the corresponding $S$-adic sequence.

	\section{Perspectives}
	We would like to finish this article by discussing some open questions that have arisen as well as some remarks.
	Recall that for $N=2$ we found, for all $x\in [0,1]\setminus\mathbb{Q}$, the optimal $C$ for which the $S$-adic sequences $\omega(x,2)$ and $\ompr(x,2)$ are $C$-balanced for each $x$. For each $N \geq 3$ we found uniform (in $x$) upper and lower bounds for such an optimal balance constant. We conjecture that the (uniform) optimal balance constant equals the upper bound.
	
We also established better bounds for the optimal balance constant for sets of points with large $N$-continued fraction digits. If for a given $x$ all the digits are larger than or equal to $2N-2$, the bound in Theorem~\ref{prop:Nbalalg} shows that $C=2$ which is certainly optimal. 
However, optimality of $C$ is not proven for any other $x$. This raises the following questions.
\begin{questions}
Let $N \geq 3$ be fixed.
\begin{enumerate}
\item What is the smallest constant $C$ for which $\ompr(x,N)$  is $C$-balanced for each $x\in [0,1]\setminus\mathbb{Q}$?
\item Let $1<k \leq \left\lfloor \frac{K-1}{K+1-N} \right\rfloor+1$. Are there $x\in W_{K,N} $ such that $\ompr(x,N)$ is $k$-balanced but not $(k-1)$-balanced? Can one characterize these $x$?
\end{enumerate}
\end{questions}
	
	Of course we can ask analogous questions for $\omega(x,N)$.

	A natural generalization of our setting are $S$-adic sequences $(\sigma_{N_n,d_n})_{n\ge 1}$ with substitutions $\sigma_{N_n,d_n}$ with $d_n \geq N_n$ given by
	\[ 
	\sigma_{N_n,d_n} :
	\begin{cases}
		0\rightarrow 0^{d_n} 1^{N_n},\\
		1\rightarrow 0.\\		
	\end{cases}
	\]
	Here $N_n\ge 1$ are arbitrary but fixed integers. This could be related to a generalization of continued fraction expansions where the numerators correspond to the sequence $(N_n)_{n \geq 1}.$

	Another direction one can take is investigating the slow version of the $N$-continued fractions discussed in Section~\ref{sec:farey}. The $S$-adic sequences arising from $\tau_B$ and $\tau_D$ are not yet studied. Since we have the invariant measure for the map $F_N$, it would natural to continue the study  of metric properties of this map. 
	
	\section*{Acknowledgements}
	   The authors would like to thank  Val\'erie Berth\'e and Wolfgang Steiner for interesting suggestions and discussions. Moreover, they are grateful to the anonymous referee for her/his valuable comments.

	\bibliographystyle{plain}  
	\bibliography{biblio}
	
\end{document}